\documentclass[a4paper,10pt]{amsart}

\usepackage{graphicx}
\usepackage{amsmath}
\usepackage{amssymb,amsthm, stmaryrd, enumitem}
\usepackage{comment}
\usepackage{hyperref}

\newtheorem{thm}{Theorem}
\newtheorem{prop}[thm]{Proposition}
\newtheorem{lem}[thm]{Lemma}
\theoremstyle{remark}
\newtheorem{rem}{Remark}
\newtheorem{ex}{Example}

\newcommand{\sA}{\triangleleft}
\newcommand{\sB}{\diamond}

\title{From Aztec diamonds to pyramids: steep tilings}
\author[J. Bouttier, G. Chapuy, S. Corteel]{J\'er\'emie Bouttier, Guillaume Chapuy, Sylvie Corteel}

\thanks{All authors are partially funded by the Ville de Paris, projet
  \'Emergences Combinatoire \`a Paris. SC acknowledges
  support from Agence Nationale de la Recherche, grant number
  ANR-08-JCJC-0011 (ICOMB).  JB and GC acknowledge partial support
  from Agence Nationale de la Recherche, grant number ANR
  12-JS02-001-01 (Cartaplus).}

\address{J.B.: Institut de Physique Th\'eorique, CEA, IPhT, 91191 Gif-sur-Yvette, France, CNRS URA 2306 and D\'epartement de Math\'ematiques et Applications, \'Ecole normale sup\'erieure, 45 rue d'Ulm, F-75231 Paris Cedex 05}
\email{jeremie.bouttier@cea.fr}
\address{G.C.: LIAFA, CNRS et Universit\'e Paris Diderot, Case 7014, F-75205 Paris Cedex 13}
\email{guillaume.chapuy@liafa.univ-paris-diderot.fr}
\address{S.C.: LIAFA, CNRS et Universit\'e Paris Diderot, Case 7014, F-75205 Paris Cedex 13}
\email{corteel@liafa.univ-paris-diderot.fr}

\date{\today}

\begin{document}
\maketitle

\begin{abstract}
  We introduce a family of domino tilings that includes tilings of the
  Aztec diamond and pyramid partitions as special cases. These tilings
  live in a strip of $\mathbb{Z}^2$ of the form $1\leq x-y\leq2\ell$
  for some integer $\ell \geq 1$, and are parametrized by a binary
  word $w\in\{+,-\}^{2\ell}$ that encodes some periodicity conditions
  at infinity. Aztec diamond and pyramid partitions correspond
  respectively to $w=(+-)^\ell$ and to the limit case
  $w=+^\infty-^\infty$. For each word $w$ and for different types of
  boundary conditions, we obtain a nice product formula for the
  generating function of the associated tilings with respect to the
  number of flips, that admits a natural multivariate
  generalization. The main tools are a bijective correspondence with
  sequences of interlaced partitions and the vertex operator formalism
  (which we slightly extend in order to handle Littlewood-type
  identities). In probabilistic terms our tilings map to Schur
  processes of different types (standard, Pfaffian and periodic). We
  also introduce a more general model that interpolates between domino
  tilings and plane partitions.
\end{abstract}

\section{Introduction}

The Aztec diamond of order $\ell$ consists of all unit squares of the
square lattice that lie completely within the region $|x| + |y| \le
\ell+1$, where $\ell$ is a fixed integer. The Aztec diamond theorem
states that the number of domino tilings of the Aztec diamond of order
$\ell$ is $2^{\ell(\ell+1)/2}$~\cite{EKLP1992} (Figure~\ref{fig0}
displays two among the 1024 domino tilings of the Aztec diamond of
order $4$). A more precise result~\cite{EKLP1992b} states that the
generating polynomial of these tilings with respect to the minimal
number of flips needed to obtain a tiling from the one with all
horizontal tiles is $\prod_{i=1}^\ell (1+q^{2i-1})^{\ell+1-i}$.  See
below for the definition of a flip.

\begin{figure}[htpb]
  \centering
  \includegraphics[width=.5\textwidth]{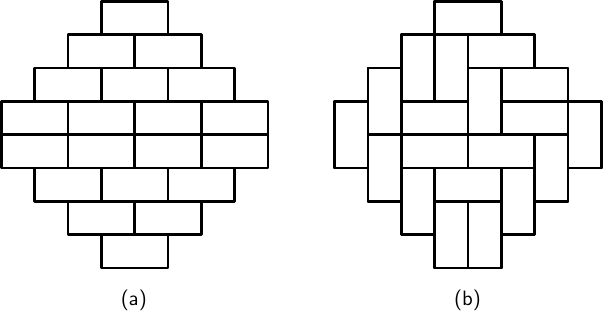}
  \caption{Two domino tilings of the Aztec diamond of size 4: the
    minimal tiling (a), and another tiling (b).}
  \label{fig0}
\end{figure}

\begin{figure}[htpb]
  \centering
  \includegraphics[scale=0.45, page=8]{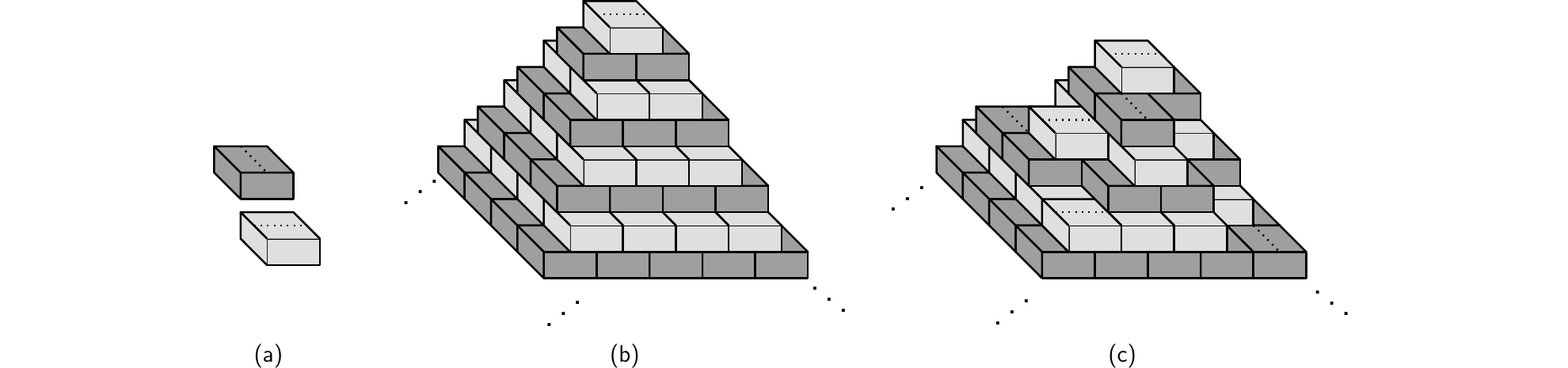}
  \caption{(a) The two types of bricks used in the construction of pyramid
partitions; (b) the fundamental pyramid partition, from which all others are
obtained by removing some bricks; (c) a pyramid partition.}
  \label{fig1}
\end{figure}

\begin{figure}[htpb]
  \centering
  \includegraphics[scale=.45,
page=8]{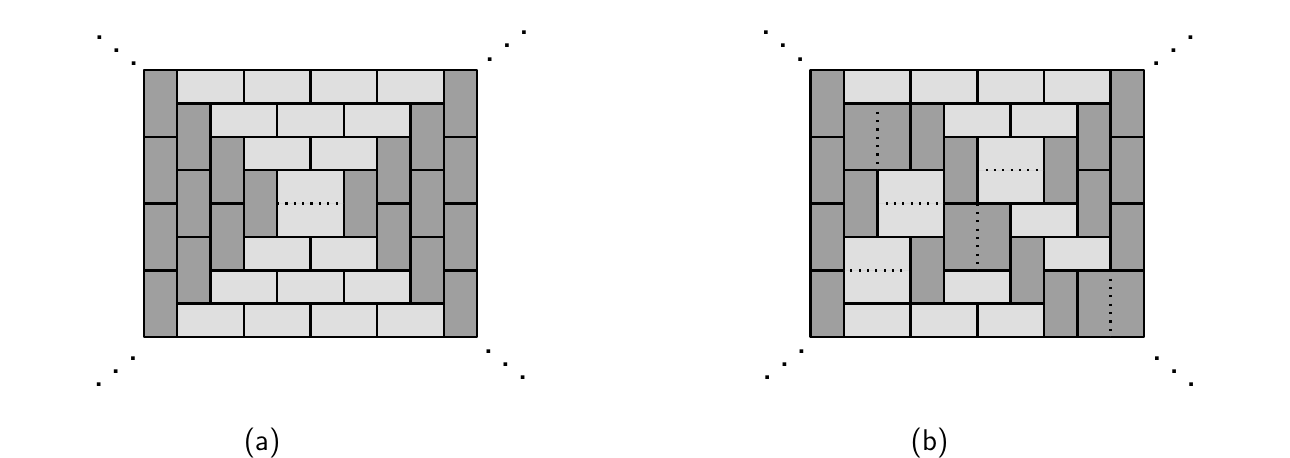}
  \caption{(a) The fundamental pyramid partition viewed from top. It induces a
domino tiling of the plane. (b) The pyramid partition of Figure~\ref{fig1}(c)
viewed from top.}
  \label{fig3}
\end{figure}

A pyramid partition is an infinite heap of bricks of size
$2\times2\times1$ in $\mathbb{R}^3$, as shown on Figure \ref{fig1}. A
pyramid partition has a finite number of maximal bricks and each brick
rests upon two side-by-side bricks, and is rotated 90 degrees from the
bricks immediately below it.  The {\em fundamental pyramid partition}
is the pyramid partition with a unique maximal brick. We denote by
$a_n$ the number of pyramid partitions obtained from the fundamental
pyramid partition after the removal of $n$ bricks. Kenyon
\cite{Kenyon2005} and Szendr{\H{o}}i \cite{S2008} conjectured the
beautiful formula
\begin{equation}
  \label{eq:pyrgf}
  \sum_n a_n q^n=\prod_{k\ge
    1}\frac{(1+q^{2k-1})^{2k-1}}{(1-q^{2k})^{2k}}
\end{equation}
which was proved by Young in two ways: first by using domino shuffling
algorithm \cite{Young:pyramid}, then by using vertex operators
\cite{Young:orbifolds}.

Aztec diamond and pyramid partitions are closely related. Indeed a
pyramid partition can be seen as a domino tiling of the whole plane,
see Figure \ref{fig3}. In this setting the removal of a brick
corresponds to the flip of two dominos.  The goal of this paper is to
show that these are indeed part of the same family of tilings that we
call \emph{steep tilings}. For good measure we also encompass the
so-called plane overpartitions \cite{CSV}.

Informally speaking, steep tilings are domino tilings of an oblique
strip, i.e.\ a strip tilted by $45^\circ$ in the square lattice, that
satisfy a ``steepness'' condition which amounts to the presence of
``frozen regions'' (more precisely, periodic repetitions of the same
pattern) sufficiently far away along the strip in both infinite
directions. We may furthermore consider different types of boundary
conditions along the two rims of the strip, namely pure, free, mixed
and periodic boundary conditions (Aztec diamonds and pyramids
partitions corresponding to the pure case). For a given asymptotic
pattern and for each type of boundary conditions, we are able to
derive an elegant product formula for the generating function of the
associated tilings (in the case of pure boundary conditions it can be
interpreted as a hook-length type formula). Our derivation is based on
a bijection between steep tilings and sequences of interlaced
partitions, and the vertex operator formalism which allows for
efficient computations. Let us mention that, in probabilistic terms,
the sequences of interlaced partitions that we encounter form Schur
processes, either in their original form
\cite{OkounkovReshetikhin:schurProcess} for pure boundary conditions,
in their Pfaffian variant \cite{BR05} for mixed boundary condition, in
their periodic variant \cite{B2007} for periodic boundary conditions,
or finally in a seemingly new ``reflected'' variant for (doubly) free
boundary conditions. The present article however focuses on the
combinatorial results, and probabilistic implications will be explored
in subsequent papers.

We now present the organization of the
paper. Section~\ref{sec:tilings} is devoted to the basic definitions
of steep tilings (Section~\ref{sec:til}), boundary conditions and
flips (Section~\ref{sec:boundary}), which we need to state our main
results (Section~\ref{sec:main}). In Section~\ref{sec:bijection}, we
discuss the bijections between steep tilings and other combinatorial
objects: particle configurations (Section~\ref{sec:particles}),
sequences of integer partitions (Section~\ref{sec:interpar}) and
height functions (Section~\ref{sec:hfun}). In Section
\ref{sec:specialcases}, we highlight some special cases: domino
tilings of Aztec diamond tilings (Section~\ref{sec:aztec}), pyramid
partitions (Section~\ref{sec:pyramids}) and plane overpartitions
(Section~\ref{sec:overpart}). In Section~\ref{sec:vertex} we compute
generating functions of steep tilings via the vertex operator
formalism: we first treat the case of the case of so-called pure
boundary conditions (Section~\ref{sec:vertexpure}) before discussing
arbitrary prescribed boundary conditions (Section~\ref{sec:vertexfix})
and finally free boundary conditions (Section~\ref{sec:vertexpure}),
where we provide a new vertex operator derivation of the Littlewood
identity. In Section~\ref{sec:cyl} we consider the case of periodic
boundary conditions (i.e.\ cylindric steep tilings). In
Section~\ref{sec:extendedmodel}, we define a more general model that
interpolates between steep tilings and (reverse) plane partitions:
after the definition of the model in terms of matchings
(Section~\ref{sec:admmatch}), we turn to the discussion of pure
boundary conditions and flips (Section~\ref{sec:extpure}) before the
enumerative results (Section~\ref{sec:extenum}); we then reinterpret
the extended model in terms of tilings
(Section~\ref{sec:extendedtilings}) and explain how it specializes to
steep tilings (Section~\ref{sec:extreduc}) and plane partitions
(Section~\ref{sec:extpp}). Concluding remarks are gathered in Section
\ref{sec:conc}.

\section{Steep tilings of an oblique strip, and the main results}
\label{sec:tilings}

In this section we define steep tilings and present our main
results. Several notions given here will become more transparent in
Section~\ref{sec:bijection} where we will study in greater depth the
structure of these objects, and in particular their connection with
height functions.

\subsection{Steep tilings of an oblique strip}
\label{sec:til}

\begin{figure}[htpb]
  \centering
  \includegraphics[width=.5\textwidth]{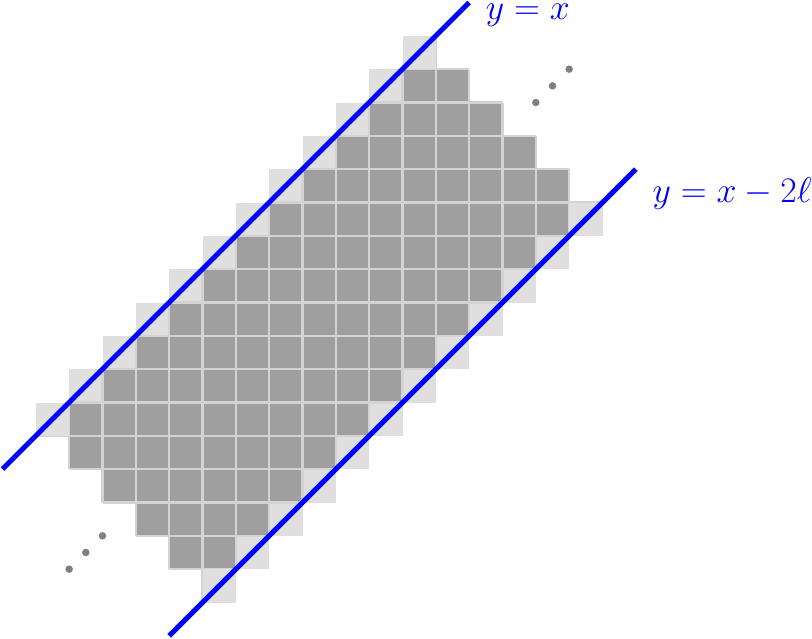}
  \caption{A portion of the oblique strip of width $2\ell=10$. The
    region $R$ covered by dominos necessarily includes the whole
    darker region, and may or may not include the lighter boundary
    squares. }
  \label{fig:obliqueStrip}
\end{figure}

Let us start by describing the general family of domino tilings we are
interested in. Recall that a \emph{domino} is a $2 \times 1$
(horizontal domino) or $1 \times 2$ (vertical domino) rectangle whose
corners have integer coordinates. Fix a positive integer $\ell$, and
consider the \emph{oblique strip} of width $2\ell$ which is the
region of the Cartesian plane comprised between the lines $y=x$ and
$y=x-2\ell$.  A \emph{tiling} of the oblique strip is a set of dominos
whose interiors are disjoint, and whose union $R$, which we call the
\emph{tiled region}, is ``almost'' the oblique strip in the sense that
\begin{equation}
  \label{eq:Rconst}
  \{ (x,y) \in \mathbb{R}^2, |x-y-\ell| \leq \ell-1 \} \subset R
  \subset \{ (x,y) \in \mathbb{R}^2, |x-y-\ell| \leq \ell+1 \},
\end{equation}
see Figure~\ref{fig:obliqueStrip}.  We are forced to use this slightly
unusual definition for a tiling since the oblique strip itself clearly
cannot be obtained as an union of dominos, as it is not a union of
unit squares with integer corners.  Observe that $R$ necessarily
contains every point of the oblique strip with integer coordinates.

\begin{figure}[htpb]
  \centering
  \includegraphics[width=\textwidth]{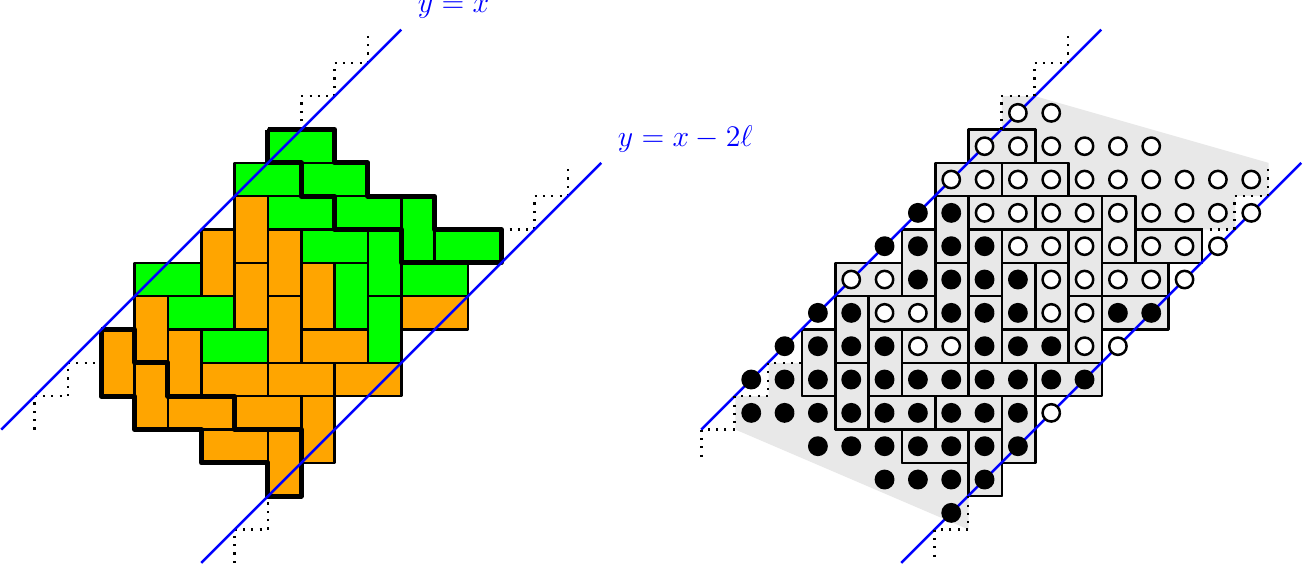}
  \caption{Left: a steep tiling of the oblique strip of width
    $2\ell=10$. North- and east-going (resp.~south- and west-going)
    dominos are represented in green (resp.~orange).  Outside of the
    displayed region, the tiling is obtained by repeating the
    ``fundamental patterns'' surrounded by thick lines. Right: the
    associated particle configuration, as defined in
    Section~\ref{sec:bijection}.}
  \label{fig:steepTiling}
\end{figure}

Following a classical terminology~\cite{CEP1996}, we say that a
horizontal (resp.~vertical) domino is \emph{north-going}
(resp.~\emph{east-going}) if the sum of the coordinates of its top
left corner is odd, and \emph{south-going} (resp.~\emph{west-going})
otherwise. We are interested in tilings of the oblique strip which are
\emph{steep} in the following sense: going towards infinity in the
north-east (resp.~south-west) direction, we eventually encounter only
north- or east-going (resp.~south- or west-going)
dominos. Figure~\ref{fig:steepTiling} displays an example of such a
tiling. The reason for which we use the term ``steep'' is that the
associated height functions (to be defined in Section~\ref{sec:hfun})
grow eventually at the maximal possible slope.

Note that any domino covering a square crossed by the boundary $y=x$
(resp.~$y=x-2\ell$) is either north- or east-going (resp.~south- or
west-going) and thus, sufficiently far away in the south-west
direction (resp.~north-east direction), all such squares are
uncovered. A further property of steep tilings is that they are
eventually periodic in both directions, as expressed by the following
proposition.

\begin{prop}
  \label{prop:periprop}
  Given a steep tiling of the oblique strip of width $2\ell$, there
  exists a unique word $w=(w_1,\ldots,w_{2\ell})$ on the alphabet
  $\{+,-\}$ and an integer $A$ such that, for all $k \in
  \{1,\ldots,\ell\}$, the following hold:
  \begin{itemize}
  \item for all $x>A$, $(x,x-2k)$ is the bottom right corner of a
    domino which is north-going if $w_{2k-1}=+$ and east-going if
    $w_{2k-1}=-$,
  \item for all $x<-A$, $(x,x-2k+2)$ is the top left corner of a
    domino which is west-going if $w_{2k}=+$ and south-going if
    $w_{2k}=-$.
  \end{itemize}
\end{prop}

\begin{proof}
  Pick $A'$ large enough so that the region $x>A'$ only contains
  north- or east-going dominos and the region $x<-A'$ only contains
  south- or west-going dominos. For $k \in \{1,\ldots,\ell\}$ and
  $x>A'$, consider the unit square with bottom right corner
  $(x,x-2k)$. It is necessarily included in the tiled region by
  \eqref{eq:Rconst}, and thus covered either by a north-going or an
  east-going domino. We then set $w_{2k-1}^{(x)}=+$ in the former case
  and $w_{2k-1}^{(x)}=-$ in the latter. Observe that we cannot have
  $w_{2k-1}^{(x)}=-$ and $w_{2k-1}^{(x+1)}=+$ as otherwise the
  corresponding dominos would overlap. This ensures that the sequence
  $(w_{2k-1}^{(x)})_{x>A'}$ is eventually constant, with value
  $w_{2k-1} \in \{+,-\}$. Similarly, by considering the unit square
  with top left corner $(x,x-2k+2)$, $x<-A'$, we define a sequence
  $(w_{2k}^{(x)})_{x<-A'}$ which is eventually constant with value
  $w_{2k}$. The proposition follows by taking $A$ large enough.
\end{proof}

\begin{ex}
  The steep tiling of Figure~\ref{fig:steepTiling} corresponds to the
  word $w=(+++++---++)$.
\end{ex}

The word $w$ of Proposition~\ref{prop:periprop} is called the
\emph{asymptotic data} of the steep tiling. We denote by
$\mathcal{T}_w$ the set of steep tilings of asymptotic data $w$,
considered up to translation along the direction $(1,1)$.

\subsection{Boundary conditions and flips}
\label{sec:boundary}

Let us now introduce a few further definitions needed to state our
main results. First, we discuss the different types of ``boundary
conditions'' that we may impose on steep tilings. What we call
boundary conditions corresponds actually to the shape of the tiled
region, since by \eqref{eq:Rconst} only the unit squares centered on
the lines $y=x$ and $y=x-2\ell$ are in an unspecified
(covered/uncovered) state, see again
Figure~\ref{fig:obliqueStrip}. Recall that the steepness condition
imposes that all unit squares centered on the line $y=x$ (resp.\
$y=x-2\ell$) are eventually uncovered when going towards infinity in
the south-west (resp.\ north-east) direction, and conversely are
eventually covered in the opposite direction (as a consequence of
Proposition~\ref{prop:periprop}). When we impose no further
restriction on the shape, we say that we have \emph{free boundary
  conditions}. A steep tiling is called \emph{pure} if there is no
``gap'' between uncovered squares on each of the two boundaries, i.e.\
if there exists two half-integers $a,b$ such that the following two
conditions hold:
\begin{itemize}
\item[(a)] the unit square centered at $(x,x)$ is covered if $x\geq a$ and
uncovered if $x<a$,
\item[(b)] the unit square centered at $(x,x-2\ell)$ is covered if
  $x\leq b$ and uncovered if $x> b$.
\end{itemize}
In this case we say of course that we have \emph{pure boundary
  conditions}.  We will see (Remark~\ref{rem:path} page
\pageref{rem:path}) that the quantity $b-a$ is determined by the
word~$w$.  We denote by $\mathcal{T}^0_w$ the set of pure steep
tilings of asymptotic data $w$, considered up to translation along the
direction $(1,1)$. We may also have \emph{mixed boundary conditions}
if (a) holds but not (b), or vice-versa. Finally, \emph{periodic
  boundary conditions} corresponds to the case where the shape of the
tiled region is such that the two boundaries ``fit'' into one another
(precisely, there exists an integer $c$ such that, for each
half-integer $x$, the unit square centered at $(x,x)$ is covered if
and only if the unit square centered at $(x+c,x+c-2\ell)$ is
uncovered). Upon identifying the two boundaries, we obtain a
\emph{cylindric steep tiling}.

We now introduce the notion of flip. A~\emph{flip} is the operation
which consists in replacing a pair of horizontal dominos forming a
$2\times 2$ block by a pair of vertical dominos, or vice-versa. A flip
can be horizontal-to-vertical or vertical-to-horizontal with obvious
definitions.  We say that the flip is centered on the $k$-th diagonal
if the center of the $2\times 2$ block lies on the diagonal $y=x-k$,
for $0< k < 2\ell$. In the case of free boundary conditions, we also
consider \emph{boundary flips} centered on the $0$-th or on the
$2\ell$-th diagonals, where only one domino covering a boundary square
is rotated (see Figure~\ref{fig:flips} below) and where the shaped of
the tiled region is modified. In the case of periodic boundary
conditions, boundary flips must be performed simultaneously on both
sides in order to preserve periodicity (see the discussion in
Section~\ref{sec:cyl}). In all cases, a vertical-to-horizontal flip
centered on the $k$-th diagonal with $k$ even and a
horizontal-to-vertical flip centered on the $k$-th diagonal with $k$
odd are called \emph{ascendent}, other flips being called
\emph{descendent}.  We will see in the next section that for each word
$w\in\{+,-\}^{2\ell}$, there exists a unique element of
$\mathcal{T}^0_w$, called the \emph{minimal tiling}, such that every
element of $\mathcal{T}_w$ (resp.\ $\mathcal{T}^0_w$) can be obtained
from it using a sequence of ascendent flips (resp.\ ascendent
non-boundary flips). Such sequences turn out to have the smallest
possible length among all possible sequences of flips between the
minimal tiling and the tiling at hand, and furthermore for each $0
\leq k \leq 2\ell $ the number of flips centered on the $k$-th
diagonal is independent of the chosen sequence.
 
\subsection{Main results} 
\label{sec:main}

We are now ready to state our main theorems. We first treat the
simplest case of pure tilings.

\begin{thm}\label{thm:main}
Let $w\in\{+,-\}^{2\ell}$ be a word. Let $T_w(q)$
be the generating function of pure steep tilings of asymptotic data $w$, 
where the exponent of $q$ records the minimal number of
flips needed to obtain a tiling from the minimal one. Then one has
\begin{equation}
  \label{eq:main}
  T_w(q) = \prod_{\substack{i<j\\ w_i=+,\ w_j=-\\ i-j \text{ odd}}}
  \left(1+q^{j-i}\right) \prod_{\substack{i<j\\ w_i=+,\ w_j=-\\ i-j
    \text{ even}}} \frac{1}{1-q^{j-i}}.
\end{equation}
\end{thm}
\begin{thm}\label{thm:mainWithStanleyWeights}
Let $w\in\{+,-\}^{2\ell}$ be a word. Let $T_w\equiv T_w(x_1,\dots,x_{2\ell-1})$
be the generating function of pure steep tilings of asymptotic data $w$,
where the exponent of the variable $x_i$ records 
the number of flips centered on the $i$-th diagonal 
in a shortest sequence of flips from the
minimal tiling. Then one has:
\begin{equation}
  \label{eq:mainWithStanleyWeights}
  T_w = \prod_{\substack{i<j\\ w_i=+,\ w_j=-\\ i-j \text{ odd}}}
  \left(1+x_ix_{i+1}\dots x_{j-1}\right)
  \prod_{\substack{i<j\\ w_i=+,\ w_j=-\\ i-j \text{ even}}}
  \frac{1}{1-x_ix_{i+1}\dots x_{j-1}}.
\end{equation}
\end{thm}
\noindent Of course Theorem~\ref{thm:main} is a direct consequence of
Theorem~\ref{thm:mainWithStanleyWeights} (which will be proved in
Section~\ref{sec:vertexpure}), by letting $x_i=q$ for each $i$.

Note that these theorems are hook formulas. Indeed, given a word $w$,
one can form the Young diagram $\lambda(w)$ delimited by the path
whose $i^{th}$ step is south if $w_i=+$ and west otherwise. Then
Theorem~\ref{thm:main} exactly states that
\begin{equation}
  T_w(q)=\prod_{c\in  \lambda(w)} (1+\epsilon(c)q^{h(c)})^{\epsilon(c)},
\end{equation}
where the product is over all cells of $\lambda(w)$, $h(c)$ denotes
the hook length of the cell $c$ (i.e.\ the number of cells to the
right or under $c$ plus one) and $\epsilon(c)=(-1)^{h(c)+1}$.  In
particular this shows that one can remove the $-$'s (resp.~the $+$'s)
placed at the beginning (resp.\ the end) of the word $w$ without
changing the value of the generating function, since this does not
change the shape of the Young diagram.  This fact is easily
interpreted geometrically, as one can check that such letters induce
regions of the oblique strip where the tiling is entirely fixed in all
configurations.  Note also that the odd hooks give a term to the
numerator of $T_w$ and the even hooks a term to the denominator of
$T_w$. Now, it is easily seen that for any $\ell$ the only Young
diagram with $\ell$ parts and only odd hooks is the staircase shape
$(\ell,\ell-1,\ldots ,2,1)$.  Therefore for a given $\ell$, the only
family of pure steep tilings whose generating function is a polynomial
is the one with asymptotic data $w=(+-)^\ell$ (upon removing the
possible trivial leading $-$'s and and trailing $+$'s). As we will see
in Section~\ref{sec:aztec}, this corresponds to tilings of the Aztec
diamond of size $\ell$ (and as a consequence of
Theorem~\ref{thm:mainWithStanleyWeights} we recover a formula due to
Stanley, see Remark~\ref{rem:stanley} page \pageref{rem:stanley}).

We then deal with the free and periodic boundary conditions. Here we
only state the univariate analogues of Theorem~\ref{thm:main}, but
multivariate formulas analogous to that of
Theorem~\ref{thm:mainWithStanleyWeights} are given in the respective
Sections~\ref{sec:vertexfree} and \ref{sec:cyl}. Let us introduce the
shorthand notation
\begin{equation}
  \label{eq:phidef}
  \varphi_{i,j}(x) =
  \begin{cases}
    1 + x & \text{if $j-i$ is odd,}\\
    1/(1-x) & \text{if $j-i$ is even.}
  \end{cases}
\end{equation}
so that \eqref{eq:main} may be rewritten in the simpler form
\begin{equation}
  T_w(q) = \prod_{\substack{i<j\\ w_i=+,\ w_j=-}} \varphi_{i,j}(q^{j-i}).
\end{equation}

\begin{thm}
  \label{thm:mainfree}
  Let $w\in\{+,-\}^{2\ell}$ be a word. Let $F_w(q)$ and $M_w(q)$ be
  the generating functions of steep tilings of asymptotic data $w$ and
  with respectively free and mixed (pure-free) boundary conditions,
  where the exponent of $q$ records the minimal number of flips needed
  to obtain a tiling from the minimal one. Then one has
  \begin{equation}
    \label{eq:mainmixed}
    M_w(q) = T_w(q) \prod_{i:\, w_i=+} \frac{1}{1-q^{m_i}}
    \prod_{\substack{i < j \\ w_i=w_j=+}} \varphi_{i,j}(q^{m_i+m_j})
  \end{equation}
  and
  \begin{equation}
    \label{eq:mainfree}
    F_w(q) = T_w(q) \prod_{k=0}^\infty \left( \frac{1}{1-q^{(k+1)L}}
        \prod_i \frac{1}{1-q^{kL+m_i}}
        \prod_{i < j} \varphi_{i,j}(q^{2kL+m_i+m_j}) \right)
  \end{equation}
  where $T_w(q)$ and $\varphi_{i,j}(\cdot)$ are as above, and where we
  use the further shorthand notations $L=2\ell+1$ and
  \begin{equation}
    \label{eq:midef}
    m_i =
    \begin{cases}
      2\ell+1-i & \text{if $w_i=+$} \\
      i & \text{if $w_i=-$}
    \end{cases}
    \qquad (i=1,\ldots,2\ell).
  \end{equation}
\end{thm}

\begin{thm}
  \label{thm:maincyl}
  Let $w\in\{+,-\}^{2\ell}$ be a word containing at least one $+$ and
  one $-$. Let $C_w(q)$ be the generating function of cylindric steep
  tilings of asymptotic data $w$, where the exponent of $q$ records
  the minimal number of flips needed to obtain a tiling from the
  minimal one. Then one has
  \begin{equation}
    \label{eq:cylgfq}
    C_w(q) = T_w(q) \prod_{k=1}^\infty \left( \frac{1}{1-q^{2k\ell}}
      \prod_{\substack{i,j\\ w_i=+,\ w_j=-}} \varphi_{i,j}(q^{2k\ell+j-i}) \right).
  \end{equation}
\end{thm}

\noindent Those two theorems will be proved in the respective
Sections~\ref{sec:vertexfree} and \ref{sec:cyl}.

\section{The fundamental bijection}
\label{sec:bijection}

The purpose of this section is to establish a general bijection
between steep tilings and sequences of interlaced partitions. The
connection is best visualized by introducing particle configurations
as an intermediate step. We will also discuss other avatars of the
same objects, namely height functions, which are convenient to
understand the flips.

\subsection{Particle configurations}
\label{sec:particles}

A \emph{site} is a point $(x,y) \in (\mathbb{Z}+\frac{1}{2})^2$
(i.e.~the center of a unit square with integer corners) such that $0
\leq x - y \leq 2\ell$. Each site may be occupied by zero or one
particle (graphically, we represent empty and occupied sites by the
respective symbols $\circ$ and $\bullet$).

Given a steep tiling of the oblique strip, we define a \emph{particle
  configuration} as follows. If a site is covered by a north- or
east-going domino, or if it is uncovered and belongs to the line
$y=x-2\ell$, then we declare it empty. Conversely, if a site is
covered by a south- or west-going domino, or if it is uncovered and
belongs to the line $y=x$, then we declare it occupied. Condition
$\eqref{eq:Rconst}$ ensures that we have defined the state of all
sites. The convention for uncovered sites is consistent if we think of
them as being covered by ``external'' dominos. See
Figure~\ref{fig:steepTiling} for an example.

The steepness condition implies that, sufficiently far away in the
north-east direction, all sites are empty and that conversely,
sufficiently far away in the south-west direction, all sites are
occupied. In particular, if we fix an integer $m \in
\{0,\ldots,2\ell\}$, we may canonically label the occupied sites along
the diagonal $y=x-m$ by \emph{positive} integers, starting from the
``highest'' one. Their abcissae form a strictly decreasing sequence
$(x_{m;n})_{n \geq 1}$ of half-integers such that $x_{m;n}+n$ is
eventually constant. Conversely, labelling the empty sites along the
same diagonal starting from the ``lowest'' one, their abcissae form a
strictly increasing sequence $(x'_{m;n})_{n \geq 1}$ such that
$x'_{m;n}-n$ is eventually constant. Actually, since the two sequences
span two disjoint sets whose union is $\mathbb{Z}+\frac{1}{2}$, there
exists an integer $c_m$ (the ``charge'') such that
\begin{equation}
  c_m = \lim_{n \to \infty} \left( x_{m;n}+n-\frac{1}{2} \right)
  = \lim_{n \to \infty} \left( x'_{m;n}-n+\frac{1}{2} \right)
    \label{eq:xxplimrel}.
\end{equation}

By examining the rules for constructing the particle configuration, we
readily see that, for all $k \in \{1,\ldots,\ell\}$ and $n \geq 1$,
\begin{equation}
  \label{eq:xcond}
  x_{2k;n} - x_{2k-1;n} \in \{ 0,1 \}
\end{equation}
(0 corresponds to a west-going domino, 1 to a south-going one) and
\begin{equation}
  \label{eq:xpcond}
  x'_{2k-1;n} - x'_{2k-2;n} \in \{ 0,1 \}
\end{equation}
(0 corresponds to an east-going domino, 1 to a north-going one).
Proposition~\ref{prop:periprop} implies that these quantities are
eventually constant as $n \to \infty$, and by \eqref{eq:xxplimrel} we
deduce that
\begin{equation}
  \label{eq:xcondlim}
  c_{2k} - c_{2k-1} =
  \begin{cases}
    0 & \text{if $w_{2k}=+$,} \\
    1 & \text{if $w_{2k}=-$,}
  \end{cases}
\end{equation}
and
\begin{equation}
  \label{eq:xpcondlim}
  c_{2k-1} - c_{2k-2} =
  \begin{cases}
    1 & \text{if $w_{2k-1}=+$,} \\
    0 & \text{if $w_{2k-1}=-$.}
  \end{cases}
\end{equation}
Note finally that the shape of the tiled region is entirely coded by
the states (empty or occupied) of the sites along the diagonals $y=x$
and $y=x-2\ell$. In particular, the tiling is pure (as defined in
Section~\ref{sec:boundary}) if and only if $x_{0;n}=x_{0;n+1}+1$ and
$x_{2\ell;n}=x_{2\ell;n+1}+1$ for all $n \geq 1$.

\subsection{Sequences of interlaced partitions}
\label{sec:interpar}

The particles along a diagonal $y=x-m$ form a so-called ``Maya
diagram''~\cite{MJD}, which classically codes an integer partition
$\lambda^{(m)}$ via
\begin{equation}
  \label{eq:lambdadef}
  \lambda^{(m)}_n = x_{m;n} + n - \frac{1}{2} - c_m
\end{equation}
(indeed the sequence $(\lambda^{(m)}_n)_{n \geq 1}$ thus defined 
is clearly a nonincreasing sequence of integers which vanishes eventually, 
i.e.\ an integer partition). Empty
sites code for the conjugate partition via
\begin{equation}
  \label{eq:lambdapdef}
  (\lambda^{(m)})'_n = - x'_{m;n} + n - \frac{1}{2} + c_m.
\end{equation}

As a straightforward consequence of the relations \eqref{eq:xcond},
\eqref{eq:xpcond}, \eqref{eq:xcondlim} and \eqref{eq:xpcondlim}, we
find that, for all $k \in \{1,\ldots,\ell\}$ and $n \geq 1$,
\begin{equation}
  \label{eq:vertstrip}
  \lambda^{(2k)}_n - \lambda^{(2k-1)}_n \in
  \begin{cases}
      \{ 0, 1 \} & \text{if $w_{2k}=+$,} \\
      \{ -1, 0 \} & \text{if $w_{2k}=-$,}
    \end{cases}
  \end{equation}
and
\begin{equation}
  \label{eq:horizstrip}
  (\lambda^{(2k-1)})'_n - (\lambda^{(2k-2)})'_n \in
  \begin{cases}
    \{ 0, 1 \} & \text{if $w_{2k-1}=+$,} \\
    \{ -1, 0 \} & \text{if $w_{2k-1}=-$.}
  \end{cases}
\end{equation}

Let us now recall some classical terminology about integer partitions
\cite{Stanley}.  Two partitions $\lambda$ and $\mu$ form a \emph{skew
  shape} $\lambda/\mu$ if the Young diagram of $\lambda$ contains that
of $\mu$, the remaining cells forming the \emph{skew diagram} of
$\lambda/\mu$. This skew diagram is called a \emph{horizontal (resp.\
  vertical) strip} if no two of its cells are in the same column
(resp.\ row), in that case we write $\lambda \prec \mu$ or $\mu \succ
\lambda$ (resp.\ $\lambda \prec' \mu$ or $\mu \succ' \lambda$), and we
say colloquially that $\lambda$ and $\mu$ are \emph{interlaced} (be it
horizontally or vertically).

We then immediately deduce from \eqref{eq:vertstrip} and
\eqref{eq:horizstrip} that the $\lambda^{(m)}$'s form a sequence of
interlaced partitions. More precisely, the Young diagrams of
$\lambda^{(2k)}$ and $\lambda^{(2k-1)}$ differ by a vertical strip
(which is either added or removed depending on the sign $w_{2k}$),
while those of $\lambda^{(2k-1)}$ and $\lambda^{(2k-2)}$ differ by a
horizontal strip (added or removed depending on $w_{2k-1}$). At this
stage, we should mention that the coding is not bijective since two
tilings differing by a translation along the direction $(1,1)$ yield
the same sequence of partitions. This is because the right hand sides
of \eqref{eq:lambdadef} and \eqref{eq:lambdapdef} are invariant if we
shift all particle positions by a constant (recall the definition
\eqref{eq:xxplimrel} of $c_m$). We say that a steep tiling is
\emph{centered} if it has $c_0=0$. Any steep tiling differs from a
centered one by a (unique) translation. We arrive at the following:

\begin{prop}[Fundamental bijection]
  \label{prop:tilseqbij}
  Given a word $w \in \{ +,- \}^{2\ell}$, the above construction
  defines a bijection between the set of centered steep tilings with
  asymptotic data $w$ and the set of sequences of partitions
  $(\lambda^{(0)},\ldots,\lambda^{(2\ell)})$ such that, for all $k \in
  \{1,\ldots,\ell\}$,
  \begin{itemize}
  \item $\lambda^{(2k-2)} \prec \lambda^{(2k-1)}$ if $w_{2k-1}=+$, and
    $\lambda^{(2k-2)} \succ \lambda^{(2k-1)}$ if $w_{2k-1}=-$,
  \item $\lambda^{(2k-1)} \prec' \lambda^{(2k)}$ if $w_{2k}=+$, and
    $\lambda^{(2k-1)} \succ' \lambda^{(2k)}$ if $w_{2k}=-$.
  \end{itemize}
  Furthermore, the bijection has the following properties.
  \begin{itemize}
  \item[A.] The precise shape of the tiled region is determined by
    the initial and final partitions $\lambda^{(0)}$ and
    $\lambda^{(2\ell)}$. In particular, the tiling is pure if and only
    if $\lambda^{(0)}=\lambda^{(2\ell)}=\emptyset$.
  \item[B.] For $m=1,\ldots,2\ell$, the absolute value of
    $|\lambda^{(m)}|-|\lambda^{(m-1)}|$ counts the number of dominos
    whose centers are on the line $y=x-m+1/2$ and whose orientations
    are opposite to the asymptotic one, see Table~\ref{tab:orient}.
  \item[C.] For $m=0,\ldots,2\ell$, $|\lambda^{(m)}|$ counts the
    number of flips centered on the $m$-th diagonal in any minimal
    sequence of flips between the tiling at hand and the minimal
    tiling $T_{\mathrm{min}}^w$ corresponding to the sequence
    $(\emptyset,\emptyset,\ldots,\emptyset)$.
  \end{itemize}
\end{prop}

\begin{table}[htpb]
  \begin{center}
    \begin{tabular}{|c|c|c|}
      \hline
      Parity of $m$ & $w_m$ & Orientation (type) \\
      \hline
      odd & $+$ & vertical (east-going)\\
      odd & $-$ & horizontal (north-going)\\
      even & $+$ & horizontal (south-going)\\
      even & $-$ & vertical (west-going)\\
      \hline
    \end{tabular}
    \bigskip
  \end{center}
  \caption{The absolute value of the difference between the sizes of $\lambda^{(m-1)}$ and of $\lambda^{(m)}$ is equal to the number of dominos whose centers are on the $y=x-m+1/2$ and have the above orientations and types.}
  \label{tab:orient}
\end{table}

\begin{proof}
  To prove the bijectivity of the construction, we simply exhibit the
  inverse mapping, and leave the reader check the details. From
  \eqref{eq:xcondlim} and \eqref{eq:xpcondlim}, we may recover the
  value of $c_m$ for all $m \in \{1,\ldots,2\ell\}$ starting with the
  data of $c_0=0$ and $w$. By \eqref{eq:lambdadef} and
  \eqref{eq:lambdapdef}, the sequence of partitions then determines
  all positions of occupied and empty sites in the particle
  configuration. Finally, the domino positions may be read off
  \eqref{eq:xcond} and \eqref{eq:xpcond}.

  The properties A and B are easy to check. The property C results
  from the discussion of height function of the forthcoming
  subsection.
\end{proof}

\begin{ex}
  The steep tiling of Figure~\ref{fig:steepTiling} corresponds to the
  sequence
  \begin{gather*}
      \lambda^{(0)} = \lambda^{(1)}= \lambda^{(2)}=
      \lambda^{(3)}=\lambda^{(4)}=\lambda^{(5)}=(1,1), \\
      \lambda^{(6)}=\lambda^{(7)}=\lambda^{(8)}=\emptyset, \quad
      \lambda^{(9)}=(1), \quad \lambda^{(10)}=(2,1).
  \end{gather*}
 For the convenience of the reader, we provide in Figure~\ref{fig:figureAddedInProof} a rotated version of the particle configuration of Figure~\ref{fig:steepTiling}, on which the Maya diagrams of the partitions and their interlacing are easier to visualize.
\end{ex}
\begin{figure}[h]
  \centering
  \includegraphics[width=0.8\textwidth]{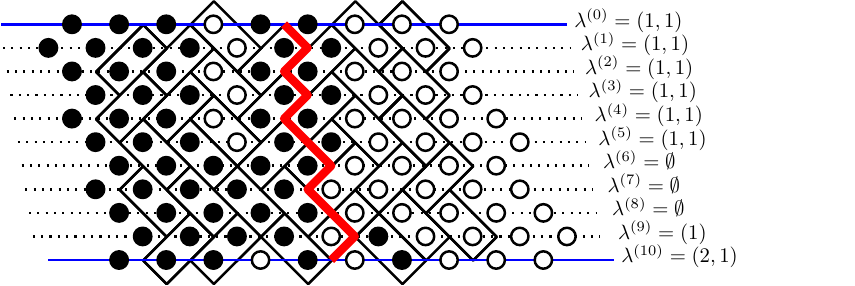}
  \caption{
The particle configuration of Figure~\ref{fig:steepTiling} after a 45$^\circ$ rotation. On each line one reads the Maya diagram of an integer partition. The thick (red) path indicates the frontier between empty and occupied sites for the minimal configuration of same asymptotic data, corresponding to the case when all partitions are empty (see also Figure~\ref{fig:minimalTiling}).
}\label{fig:figureAddedInProof}
\end{figure}

\subsection{Height functions}
\label{sec:hfun}

\begin{figure}[htpb]
  \centering
  \includegraphics[width=0.8\textwidth]{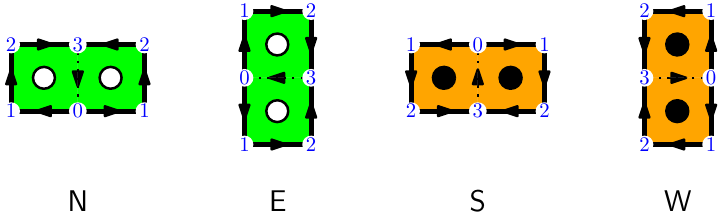}
  \caption{Variation of the height function around each type of domino
    (numbers represent the height difference with respect to the
    ``lowest'' vertex around the domino). The local particle
    configuration is displayed.}
  \label{fig:heightFunction}
\end{figure}

To complete the picture, let us discuss the height function associated
to a steep tiling of the oblique strip. Let $V$ be the set of integer
points in the oblique strip, namely
\begin{equation}
  \label{eq:Vdef}
  V = \{ (x,y) \in \mathbb{Z}^2, 0 \leq x-y \leq 2\ell \}.
\end{equation}
Consider the oriented graph with vertex set $V$ where each vertex
$(x,y) \in V$ such that $x-y$ is odd has exactly two outgoing edges
pointing to $(x \pm 1,y)$, and exactly two incoming edges originating
from $(x,y \pm 1)$. Following Thurston~\cite{Thurston:tiling}, we
define a \emph{height function} as a function $H: V \to \mathbb{Z}$
such that, for any oriented edge $(x,y) \to (x',y')$, we have
\begin{equation}
  \label{eq:heidef}
  H(x',y') - H(x,y) \in \{ 1,-3 \}.
\end{equation}
It is easily seen that height functions (considered up to an additive
constant) are in bijection with domino tilings: each edge with a
height difference of $-3$ corresponds to a domino, see
Figure~\ref{fig:heightFunction}. Note that boundary dominos
(i.e.~dominos overlapping either the line $y=x$ or the line
$y=x-2\ell$) have one of their corners not in $V$, but no information
is lost in restricting the height function to $V$.

\begin{figure}[htpb]
  \centering
  \includegraphics[width=.5\textwidth]{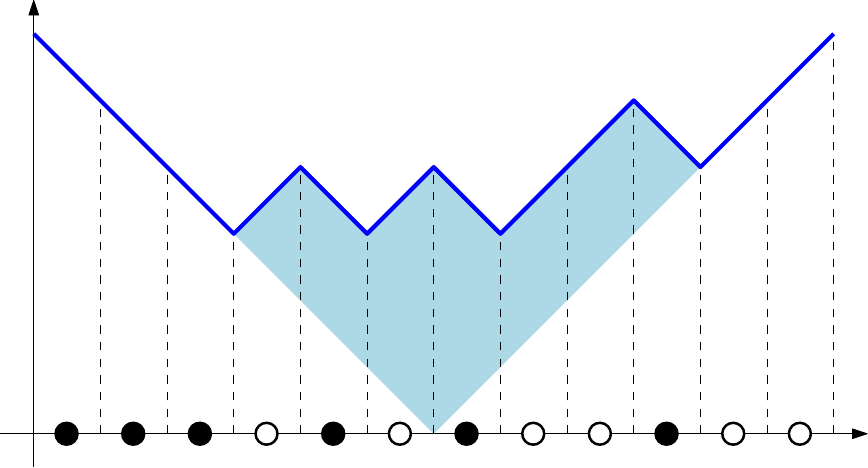}
  \caption{Russian representation of the partition $(4,2,1)$, and the
    associated Maya diagram.}
  \label{fig:russian}
\end{figure}

\begin{figure}[htpb]
  \centering
  \includegraphics[width=0.5\textwidth]{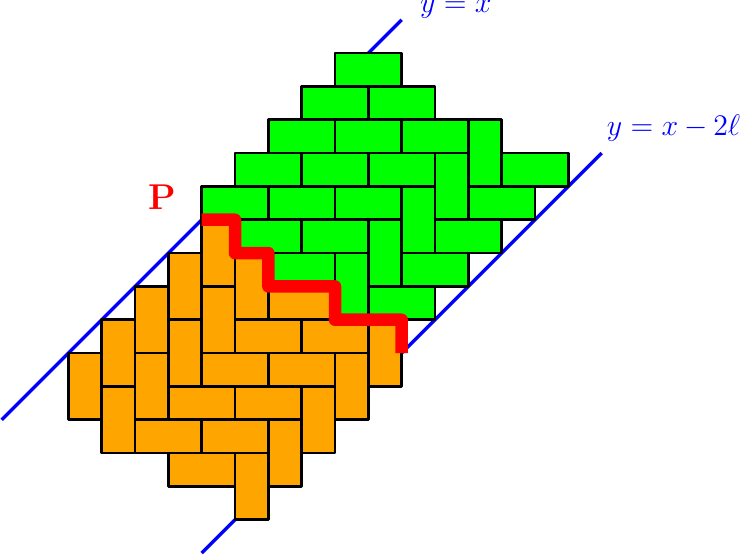}
  \caption{The minimal tiling $T_\mathrm{min}^w$ corresponding to the asymptotic data
$w=(+++++---++)$. The minimal tiling is separated into two regions,
one filled with only  south- or west- going (orange) dominos, and the other
one filled with only  north- or east- going (green) dominos, corresponding respectively to occupied
and unoccupied sites in the particle configuration.
\label{fig:minimalTiling}}
\end{figure}

Interestingly, the height function is closely related to the particle
configuration: as is apparent on Figure~\ref{fig:heightFunction}, the
height difference between the vertex at the top right and that at the
bottom left of an empty (resp.~occupied) site is $+2$
(resp.~$-2$). Hence, for any $m \in \{0,\ldots,2\ell\}$, the graph of the
function $x \mapsto H(x,x-m)$ coincides (up to scaling and
translation) with the ``Russian'' representation of the partition
$\lambda^{(m)}$, see Figure~\ref{fig:russian}. By
translating Proposition~\ref{prop:periprop} in the language of height
functions we find that, for $|x|$ large enough, we have
\begin{equation}
  \label{eq:Hasymp}
  H(x,x-m) = 2 |x-c_m| + h_m
\end{equation}
where $c_m$ is defined as in \eqref{eq:xxplimrel} and $h_m-h_{m+1}=1$
if $w_m=+$, $-1$ if $w_m=-$ (we may fix $h_0=0$ since the height
function is defined modulo an additive constant). Observe that the
height function grows eventually at the maximal possible slope, which
is why we call the corresponding tiling ``steep''. For fixed
asymptotic data, the lowest possible height function
$H_{\mathrm{min}}^w$ is achieved when \eqref{eq:Hasymp} holds for all
$x$, which corresponds to having $\lambda^{(m)}=\emptyset$ for all
$m$. The corresponding centered tiling $T_{\mathrm{min}}^w$ is the
\emph{minimal} tiling of asymptotic data $w$, see
Figure~\ref{fig:minimalTiling}.

\begin{rem}\label{rem:path}
  In the minimal tiling $T_{\mathrm{min}}^w$, both the region covered
  by north- or east-going dominos and that covered by east- or
  south-going ones are connected, and the frontier between them is
  made of a finite path $P$, see Figure~\ref{fig:minimalTiling}. By
  \eqref{eq:Hasymp}, $P$ intersects the line $x=y-m$ at the point
  \begin{equation}
    P_m = (c_m,c_m-m) = \left(
      \frac{m-\sum_{j=1}^m w_j(-1)^j}{2},
      \frac{-m-\sum_{j=1}^m w_j(-1)^j}{2}
    \right).
  \end{equation}
  In particular, $P_0$ and $P_{2\ell}$ mark the limit between
  uncovered and covered boundary squares in $T_{\mathrm{min}}^w$, thus
  in any centered pure steep tiling.
\end{rem}

\begin{figure}[htpb]
  \centering
  \includegraphics[width=\textwidth]{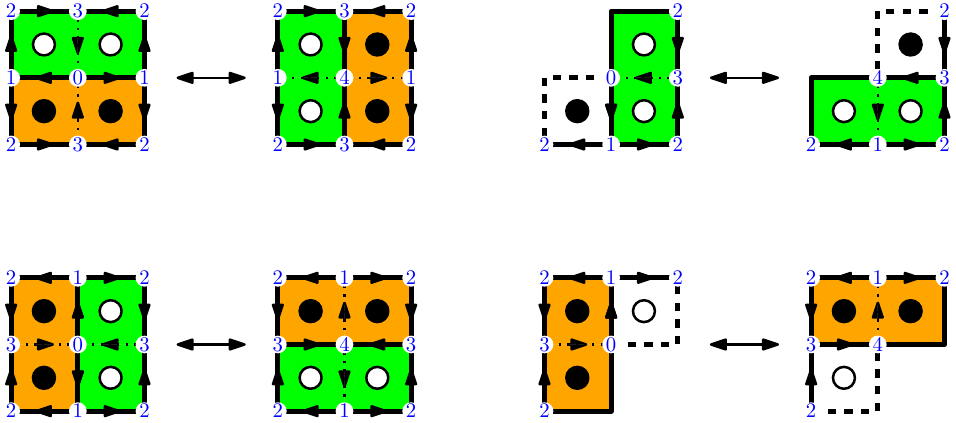}
  \caption{The effect of flips on the height function and on the
    particle configuration (left: bulk case, right: boundary
    case). The flips are ascendent from left to right and descendent
    from right to left.}
  \label{fig:flips}
\end{figure}

As pointed out by Elkies \emph{et al.}\ \cite{EKLP1992} in the context
of the Aztec diamond, and understood by Propp \cite{Propp1993} in a
much broader context, height functions play a crucial role in the
study of flips. Recall that here a flip consists in replacing a pair
of horizontal dominos forming a $2 \times 2$ block by a pair of
vertical dominos, or vice-versa. In the context of tilings of the
oblique strip, we also consider ``boundary flips'' where we rotate a
boundary domino adjacent to an uncovered square (thus changing the
shape of the tiled region). As is apparent on Figure~\ref{fig:flips},
a flip modifies the height function precisely at one vertex, where the
height is increased or decreased by 4 (this is the smallest possible
change, since the difference between two height functions is constant
modulo 4). Given a tiling $T$ with height function $H$, it may be
shown along the same lines as in~\cite{EKLP1992} that the minimal
number of flips needed to pass from $T_{\mathrm{min}}^w$ to $T$ is
equal to
\begin{equation}
  r_w(T) = \sum_{(x,y) \in V} |H(x,y)-H_{\mathrm{min}}^w(x,y)|/4.
  \label{eq:rt}
\end{equation}
By \eqref{eq:Hasymp}, this number is finite if and only if $T$ is
centered and has asymptotic data~$w$, and we then have
\begin{equation}
  \label{eq:dfliptmin}
  r_w(T) = \sum_{m=0}^{2\ell} |\lambda^{(m)}|
\end{equation}
where $(\lambda^{(0)},\ldots,\lambda^{(2\ell)})$ is the sequence of
interlaced partitions associated to $T$, and where $|\lambda^{(m)}|$
denotes the size of the partition $\lambda^{(m)}$.  By an easy
refinement of \eqref{eq:rt}, we find that $|\lambda^{(m)}|$ precisely
counts the number of flips made on the diagonal $x=y-m$ (which are
boundary flips for $m=0$ or $2\ell$) in any shortest sequence of flips
from $T_{\mathrm{min}}^w$ to $T$. This establishes the property~C in
Proposition~\ref{prop:tilseqbij}.

\section{Some particular cases}
\label{sec:specialcases}

We now discuss a few particularly interesting cases of the bijection.

\subsection{Aztec diamond}
\label{sec:aztec}

Domino tilings of the Aztec diamond of size $\ell$ are obtained by
considering the word $w=(+-)^\ell=+-+-\cdots+-$ ($\ell$ times):

\begin{prop} \label{prop:aztecc}
  There is a one-to-one correspondence between domino tilings of the
  Aztec diamond of size $\ell$ and sequences of partitions
  $(\lambda^{(0)},\ldots,\lambda^{(2\ell)})$ such that
  \begin{equation}
    \label{eq:azcase}
    \emptyset = \lambda^{(0)} \prec \lambda^{(1)} \succ' \lambda^{(2)} \prec
    \lambda^{(3)} \succ' \cdots \prec \lambda^{(2\ell-1)} \succ'
    \lambda^{(2\ell)} = \emptyset.
  \end{equation}
\end{prop}

This fact was also observed by Dan Betea (personal communication).

\begin{proof}
  \begin{figure}[htpb]
    \centering
    \includegraphics[width=.5\textwidth]{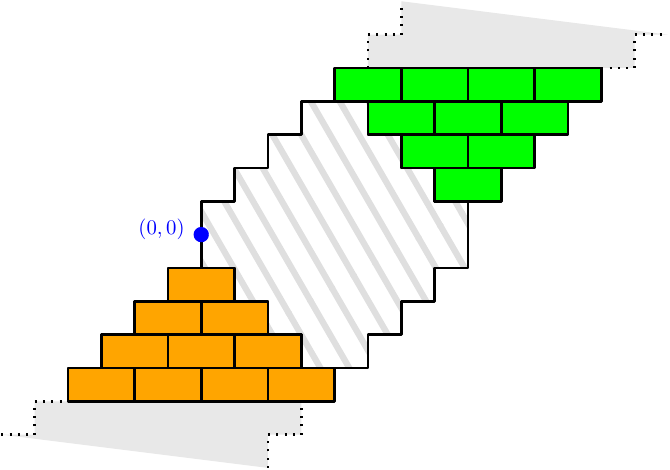}
    \caption{A tiling of the Aztec diamond of size $\ell$ (here
      $\ell=4$) may be completed into a steep tiling with asymptotic
      data $(+-)^\ell$.}
    \label{fig:aztecDiamond}
  \end{figure}

  By translating a tiling of the Aztec diamond, and completing it in
  the way displayed on Figure~\ref{fig:aztecDiamond}, we obtain a
  centered pure steep tiling with asymptotic data $w=(+-)^\ell$.  By
  Proposition~\ref{prop:tilseqbij}, it corresponds to a sequence of
  partitions satisfying \eqref{eq:azcase} with
  $\lambda^{(0)}=\lambda^{(2\ell)}=\emptyset$.

  Conversely, we need to show that a steep tiling obtained from a
  sequence of partitions satisfying \eqref{eq:azcase} is ``frozen''
  outside the Aztec diamond. While this may be checked directly by
  reasoning on dominos, let us here exploit the correspondence with
  partitions: observe that (the Young diagram of) $\lambda^{(1)}$ has
  at most one row (since $\lambda^{(1)} \succ \emptyset$), and so does
  $\lambda^{(2)}$ (since $\lambda^{(1)} \succ' \lambda^{(2)}$). Hence,
  $\lambda^{(3)}$ and $\lambda^{(4)}$ have at most two rows (since
  $\lambda^{(2)} \prec \lambda^{(3)} \succ' \lambda^{(4)}$), and so on
  we find by induction that, for all $k \in \{1,\ldots,\ell\}$,
  $\lambda^{(2k-1)}$ and $\lambda^{(2k)}$ have at most $k$
  rows. Similarly, by reasoning on \eqref{eq:azcase} backwards, we
  find that $\lambda^{(2k-1)}$ and $\lambda^{(2k-2)}$ have at most
  $\ell+1-k$ columns.
  Translating these constraints in the particle language, we find that
  all sites strictly above the line $x+y=2\ell$ are empty, and all
  sites strictly below the line $x+y=0$ are occupied. Thus, the height
  function coincides with $H_{\mathrm{min}}^w$ for $x+y \geq 2\ell$ or
  $x+y \leq 0$, so the tiling coincides in these regions with the
  minimal tiling $T_{\mathrm{min}}^w$, which here consists only of
  horizontal dominos, as wanted.
\end{proof}

\begin{rem}
  \label{rem:stanley}
  Using the property~B in Proposition~\ref{prop:tilseqbij}, it is not
  difficult to see that the multivariate weighting scheme of
  Theorem~\ref{thm:mainWithStanleyWeights} is equivalent to the
  so-called Stanley's weighing scheme for the Aztec diamond
  \cite{Propp:talk,BYYangPhD}. Indeed, attaching a weight $x_i$
  ($i=1,\ldots,2\ell-1$) to each flip centered on the $i$-th diagonal
  is tantamount to attaching a weight $z_j$ ($j=1,\ldots,2\ell$) to
  each vertical domino whose center is on the line $y=x-j+1/2$, upon
  imposing the relation
  \begin{equation}
    x_{2k-1}=z_{2k-1}z_{2k}, \qquad x_{2k}=\frac{1}{z_{2k}z_{2k+1}}.
  \end{equation}
  Thus, Theorem~\ref{thm:mainWithStanleyWeights} implies that
  \begin{equation}
    \label{eq:stanley}
    T_{(+-)^\ell} =  \prod_{\substack{1 \leq i<j \leq 2\ell \\ \text{$i$ odd, $j$ even}}} (1 + z_iz_j)
  \end{equation}
  which is equivalent to Stanley's formula (where weights for
  horizontal dominos can be set to $1$ without loss of generality).
\end{rem}

\subsection{Pyramid partitions}
\label{sec:pyramids}

They can be recovered by considering centered pure steep tilings with
asymptotic data $w=+^\ell -^\ell$ (that is, $+$ repeated $\ell$ times
then $-$ repeated $\ell$ times) and then letting $\ell\rightarrow
+\infty$. Let $\ell$ be a fixed odd integer and denote by $P_\ell$ the
set of pyramid partitions that we can obtain from the fundamental
partition given on Figure \ref{fig3} where the center of the brick on
the top is $(0,0)$ and where one can only take off bricks that lie
inside the strip $-\ell \le x-y\le \ell$ (note that, by our
conventions, the removal of a brick actually corresponds to an
ascendent domino flip).  It is straightforward to see that:
\begin{prop}[see also {\cite[Lemma~5.9]{Young:orbifolds}}]\label{pyra_ell}
  There is a one-to-one correspondence between pyramid partitions in
  $P_\ell$ and pure (centered steep) tilings with asymptotic data
  $w=(+^\ell -^\ell)$.  Equivalently, there is a bijection between
  pyramid partitions in $P_\ell$ and sequences of partitions
  $(\lambda^{(0)},\ldots,\lambda^{(2\ell)})$ such that
  \begin{equation}
    \label{eq:pyreven}
    \emptyset = \lambda^{(0)} \prec \lambda^{(1)} \prec' \lambda^{(2)} \prec
    \lambda^{(3)} \cdots \prec \lambda^{(\ell)}  \succ' \lambda^{(\ell+1)} \succ \cdots \succ \lambda^{(2\ell-1)} \succ'
    \lambda^{(2\ell)} = \emptyset.
  \end{equation}
\end{prop}

\begin{figure}[htpb]
  \centering
  \includegraphics[width=0.5\textwidth]{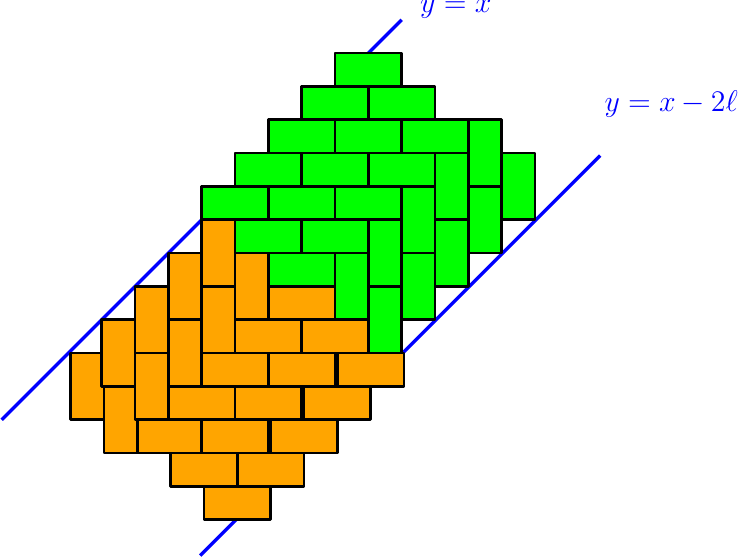}
  \caption{The minimal tiling $T_\mathrm{min}^w$ corresponding to the asymptotic data
$w=(+++++-----)$. 
\label{pyra}}
\end{figure}

Indeed, one can easily see that the empty pyramid partition restricted
to the oblique strip of width $2\ell$ is nothing but
$T_{\mathrm{min}}^{+^\ell -^\ell}$ translated by $(-\ell,\ell)$, see
Figure \ref{pyra}.  Theorem \ref{thm:main} gives
\begin{equation}
  T_{(+^\ell -^\ell)}(q)=\prod_{i=1}^{\ell}(1+q^{2i-1})^{2\min(i,\ell+1-i)-1}\prod_{i=1}^{\ell-1}\frac{1}{(1-q^{2i})^{2\min(i,\ell-i)}}.
\end{equation}
It is immediate to see that when $\ell\rightarrow +\infty$, we recover
the generating function of pyramid partitions \eqref{eq:pyrgf}. A
similar construction holds when $\ell$ is an even integer (the
``central block'' is rotated by $90^\circ$).

\subsection{Plane overpartitions}
\label{sec:overpart}

We now provide an example of non pure steep tilings. A \emph{plane
  overpartition} \cite{CSV} is a plane partition where in each row the
last occurrence of an integer can be overlined or not and in each
column the first occurrence of an integer can be overlined or not and
all the others are overlined. 
An example of a plane overpartition of shape $(4,4,2,1)$ is
\begin{equation}
  \begin{array}{llll}
    2 & 2 & \bar{2} & \bar{1} \\
    \bar{2} & 1 & 1 & \bar{1}\\
    \bar{2} & \bar{1} & &\\
    1 &  & & \\
  \end{array}
  \label{eq:overpart}
\end{equation}
It is easily seen that a plane overpartition of shape $\lambda$
containing integers at most $\ell$ is in bijection with a sequence of
interlaced partitions such that
\begin{equation}
  \label{eq:overseq}
  \emptyset=\lambda^{(0)} \prec \lambda^{(1)} \prec' \lambda^{(2)} \prec
  \lambda^{(3)} \prec' \lambda^{(4)} \prec \cdots \prec \lambda^{(2\ell-1)}
  \prec' \lambda^{(2\ell)} = \lambda.
\end{equation}
Indeed, for $i=1,\ldots,\ell$, the horizontal strip
$\lambda^{(2i-1)}/\lambda^{(2i-2)}$ (resp.\ the vertical strip
$\lambda^{(2i)}/\lambda^{(2i-1)}$) is formed by the non overlined
entries (resp.\ the overlined entries) equal to $\ell+1-i$.

\begin{figure}[htpb]
  \centering
  \includegraphics[width=0.3\textwidth]{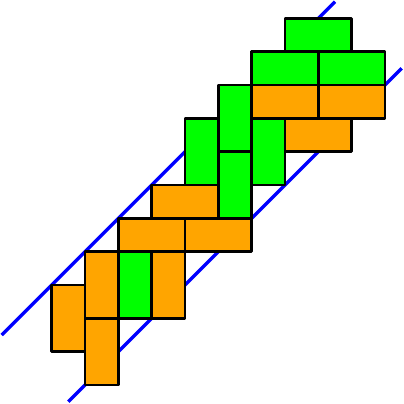}
  \caption{The steep tiling corresponding to the plane overpartition
    of Equation \eqref{eq:overpart}.}
\label{fig:overpart}
\end{figure}

By Proposition~\ref{prop:tilseqbij}, plane overpartitions are then in
bijection with some (non pure) steep tilings with asymptotic data
$+^{2\ell}$, see Figure~\ref{fig:overpart} for an example. These
tilings are essentially the same as those discussed in \cite[Section
4]{CSV}.

Note that, upon reversing and conjugating the sequence
\eqref{eq:overseq}, a plane overpartition may as well be coded by a
sequence
\begin{equation}
  \label{eq:overseqbis}
  \lambda'=\mu^{(0)} \succ \mu^{(1)} \succ' \mu^{(2)} \succ
  \mu^{(3)} \succ' \mu^{(4)} \succ \cdots \succ \mu^{(2\ell-1)}
  \succ' \mu^{(2\ell)} = \emptyset
\end{equation}
corresponding to a steep tiling with asymptotic data $-^{2\ell}$
(which is actually the previous tiling rotated by $180^\circ$).  We
then observe that pairs of plane overpartitions whose shape are
conjugate to one another are in bijection with pyramid partitions:
this is seen by concatenating their associated sequences of interlaced
partitions (upon using the convention \eqref{eq:overseq} for the first
plane overpartition and the convention \eqref{eq:overseqbis} for the
other one), or alternatively by ``assembling'' their associated domino
tilings into one another.

\begin{rem}
  As pointed out by Sunil Chhita, the case where $\lambda$ is a
  rectangular shape corresponds to domino tilings of the so-called
  double Aztec diamond \cite{AJvM14,ACJvM}. Compare for instance
  \cite[Figure 1]{AJvM14} with \cite[Figures 8 and 11]{CSV} (where one
  shall think of all outgoing red lines at the bottom being moved as
  much as possible to the right in the case where $\lambda$ is a
  rectangular shape).
\end{rem}

\section{Enumeration via the vertex operator formalism}
\label{sec:vertex}

\subsection{Pure steep tilings}
\label{sec:vertexpure}

The purpose of this section is to establish
Theorem~\ref{thm:mainWithStanleyWeights}, which implies
Theorem~\ref{thm:main} upon taking $x_i=q$ for all $i$. By
our general bijection (Proposition~\ref{prop:tilseqbij}), the
enumeration of elements of $\mathcal{T}_w$ is equivalent to the one of
interlaced sequences of partitions with empty boundary
conditions. This amounts to computing the partition function of a
\emph{Schur process}
\cite{OkounkovReshetikhin:schurProcess,Borodin:dynamics}, which is
easily done using the \emph{vertex operator formalism}, see e.g.\
\cite{Kac,Okounkov:wedge}.  Let us now briefly recall this formalism.

We work over the vector space of formal sums of partitions, with basis
$\{|\lambda\rangle, \lambda \in \Lambda\}$ and dual basis
$\{\langle\mu|, \mu \in \Lambda\}$, where $\Lambda$ is the set of all
integer partitions (here we use the convenient bra-ket notation).  We
consider the \emph{vertex operators} $\Gamma_+(t)$ and $\Gamma_-(t)$
defined by
\begin{equation}
  \begin{split}
    \Gamma_+(t) |\lambda\rangle &= \sum_{\mu:\ \mu\prec \lambda}
    t^{|\lambda|-|\mu|} |\mu\rangle, \\
    \Gamma_-(t) |\lambda\rangle &= \sum_{\mu:\ \mu\succ\lambda}
    t^{|\mu|-|\lambda|} |\mu\rangle,
  \end{split}
\end{equation}
where $t$ must be seen as a formal variable.
Following~\cite{Young:orbifolds}, we also introduce the conjugate vertex operators $\Gamma'_+(t)$ and $\Gamma'_-(t)$
defined by 
\begin{equation}
  \begin{split}
    \Gamma'_+(t) |\lambda\rangle &= \sum_{\mu:\ \mu\prec' \lambda}
    t^{|\lambda|-|\mu|} |\mu\rangle, \\
    \Gamma'_-(t) |\lambda\rangle &= \sum_{\mu:\ \mu\succ' \lambda}
    t^{|\mu|-|\lambda|} |\mu\rangle.
  \end{split}
\end{equation}
Note that $\Gamma_\pm'(t) = \omega \Gamma_\pm (t) \omega$ where
$\omega: |\lambda\rangle \longmapsto |\lambda'\rangle$ is the
conjugation of partitions.

\begin{lem}\label{lem:gfInTermsOfOperators}
  Fix a word $w\in\{+,-\}^{2\ell}$, and $\alpha,\beta$ two
  partitions. Let $T_{w,\alpha,\beta}\equiv
  T_{w,\alpha,\beta}(x_1,x_2,\dots,x_{2\ell})$ be the generating
  function of sequences $(\lambda^{(0)},\lambda^{(1)}, \dots,
  \lambda^{(2\ell)})$ of partitions that are interlaced as described
  in Proposition~\ref{prop:tilseqbij}, with $\lambda^{(0)}=\alpha$ and
  $\lambda^{(2\ell)}=\beta$, and where the exponent of the variable
  $x_i$ ($1\leq i\leq 2\ell$) records the size of the partition
  $\lambda^{(i)}$ minus that of $\alpha$.  Then $T_{w,\alpha,\beta}$
  is given by
  \begin{eqnarray}\label{eq:gfInTermsOfOperators}
    T_{w,\alpha,\beta} = \langle\alpha|
    \prod_{i=1}^\ell \Gamma_{w_{2i-1}}(y_{2i-1}^{w_{2i-1}})
    \Gamma'_{w_{2i}}(y_{2i}^{w_{2i}})
    |\beta\rangle
  \end{eqnarray}
  with $y_i=x_ix_{i+1}\dots x_{2\ell}$. Here and in the sequel,
  $y_i^{w_i}$ shall be understood as $y_i$ if $w_i=+$, and $1/y_i$ if
  $w_i=-$.
\end{lem}

\begin{proof}
  By the definition of the operators $\Gamma_\pm$ and $\Gamma'_\pm$,
  we have
  \begin{equation}
    \langle\alpha|
    \prod_{i=1}^\ell \Gamma_{w_{2i-1}}(q_{2i-1})
    \Gamma'_{w_{2i}}(q_{2i})
    |\beta \rangle
    = \sum
    \prod_{i=1}^{2\ell} q_{i}^{w_{i}(|\lambda^{(i)}|-|\lambda^{(i-1)}|)}
  \end{equation}
  where the sum runs over all sequences of partitions
  $(\lambda^{(0)},\lambda^{(1)}, \dots, \lambda^{(2\ell)})$ satisfying
  the interlacing conditions of Proposition~\ref{prop:tilseqbij}.
  Taking $q_i=y_i^{w_i}$, we obtain the correct weight, since
  $q_i^{w_i}=y_i$, and since in the product $\prod_{i=1}^{2\ell}
  (x_ix_{i+1}\dots x_{2\ell})^{|\lambda^{(i)}|-|\lambda^{(i-1)}|}$ the
  exponent of $x_i$ is $\sum_{j=1}^i
  |\lambda^{(j)}|-|\lambda^{(j-1)}|=|\lambda^{(i)}|-|\alpha|$.
\end{proof}

Vertex operators are known (see e.g.\
\cite[Lemma~3.3]{Young:orbifolds}) to satisfy the following nontrivial
commutation relations
\begin{equation}
  \label{eq:gamcom}
  \begin{split}
    \Gamma_+(t)\Gamma_-(u) &= \frac{1}{1-tu} \Gamma_-(u)\Gamma_+(t),\\
    \Gamma'_+(t)\Gamma'_-(u) &= \frac{1}{1-tu} \Gamma'_-(u)\Gamma'_+(t),\\
    \Gamma_+(t)\Gamma'_-(u) &= (1+tu) \Gamma'_-(u)\Gamma_+(t),\\
    \Gamma'_+(t)\Gamma_-(u) &= (1+tu) \Gamma_-(u)\Gamma'_+(t),
  \end{split}
\end{equation}
while other commutation relations are trivial (namely, two vertex
operators with the same sign in index commute together).  Note the
following more compact way of rewriting \eqref{eq:gamcom}: fix two
symbols $\Gamma^\sA, \Gamma^\sB\in\{\Gamma,\Gamma'\}$, then
\begin{equation}\label{eq:compactCommutationRelation}
  \Gamma^\sA_+(t)\Gamma^\sB_-(u) = (1+\epsilon tu)^\epsilon
  \,\Gamma^\sB_-(u)\Gamma^\sA_+(t),
\end{equation}
where $\epsilon = -1$ if $\Gamma^\sA=\Gamma^\sB$ and $\epsilon=+1$
otherwise.

\begin{proof}[Proof of Theorem~\ref{thm:mainWithStanleyWeights}]
  By Proposition~\ref{prop:tilseqbij}, the desired generating function
  $T_w$ of pure steep tilings is nothing but
  $T_{w,\emptyset,\emptyset}$, as expressed by
  Lemma~\ref{lem:gfInTermsOfOperators} in terms of vertex operators
  (note that $T_w$ does not depend on $x_{2\ell}$). We may now
  evaluate the right hand side of \eqref{eq:gfInTermsOfOperators}:
  note first that for any weights $z_1,\dots,z_{2\ell}$, symbols
  $(\Gamma^{\diamond_i})_{1\leq i \leq 2\ell}\in \{\Gamma,
  \Gamma'\}^{2\ell}$ and $m$ between $1$ and $2\ell$, we have
  \begin{equation}\label{eq:trivialScalarProduct}
    \langle\emptyset|
    \prod_{i=1}^m \Gamma^{\diamond_i}_{-}(z_{i})
    \prod_{i=m+1}^{2\ell}
    \Gamma^{\diamond_i}_{+}(z_{i})
    |\emptyset\rangle =1,
  \end{equation}
  since by definition $\Gamma^\diamond_+|\emptyset \rangle =
  |\emptyset\rangle$ and $\langle \emptyset | \Gamma^\diamond_- =
  \langle \emptyset |$ for any $\Gamma^\diamond\in\{\Gamma,
  \Gamma'\}$. Therefore we can
  evaluate~\eqref{eq:gfInTermsOfOperators} by ``moving'' all the
  $\Gamma_-$ and $\Gamma'_-$ operators to the left using the
  commutation relations \eqref{eq:gamcom}. We will end up with a
  multiplicative prefactor coming from the commutation relations, and
  a remaining scalar product of the
  form~\eqref{eq:trivialScalarProduct} which evaluates to $1$.

  In the process of moving these operators to the left, for each $i<j$
  such that $w_i=+$ and $w_j=-$ we have to exchange the operators
  $\Gamma^\sA_+(y_i)$ and $\Gamma^\sB_-(y_j^{-1})$ where $\Gamma^\sA$
  (resp.\ $\Gamma^\sB$) is equal to $\Gamma$ if $i$ (resp.\ $j$) is
  odd, and to $\Gamma'$ if $i$ (resp.\ $j$) is even.
  By~\eqref{eq:compactCommutationRelation}, the multiplicative
  contribution of this exchange is equal to
  \begin{equation}
    \left(1+\epsilon_{i,j} \frac{y_i}{y_j}\right)^{\epsilon_{i,j}} =
    \left(1+\epsilon_{i,j} x_{i}x_{i+1}\dots x_{j-1}\right)^{\epsilon_{i,j}},
  \end{equation}
  where $\epsilon_{i,j}=-1$ if $i$ and $j$ have the same parity, and
  $\epsilon_{i,j}=1$ otherwise. The desired expression
  \eqref{eq:mainWithStanleyWeights} follows.
\end{proof}

\begin{rem}
  As pointed out by Paul Zinn-Justin \cite{PZJBerkeley}, it is also
  possible to relate the vertex operator formalism to domino tilings
  via the six-vertex model on the free-fermion line, the product
  $\Gamma_+(t) \Gamma'_-(u)$ corresponding essentially to the transfer
  matrix of this model.
\end{rem}

\begin{rem}
  In the case of domino tilings of the Aztec diamond, the vertex
  operator computation can be related to domino shuffling
  \cite{EKLP1992b}. Indeed, the commutation relation between
  $\Gamma_+(t)$ and $\Gamma'_-(u)$ can be derived bijectively via
  domino shuffling (we leave this as a pleasant exercise to the
  reader) and from it we deduce the relation
  \begin{multline}
    \langle \emptyset | \Gamma_+(z_1) \Gamma'_-(z_2) \Gamma_+(z_3)
    \Gamma'_-(z_4) \Gamma_+(z_5) \Gamma'_-(z_6) \cdots
    \Gamma_+(z_{2\ell-1})
    \Gamma'_-(z_{2\ell}) | \emptyset \rangle = \\
    \prod_{i=1}^\ell (1+z_{2i-1}z_{2i}) \times \langle \emptyset |
    \Gamma_+(z_1) \Gamma'_-(z_4) \Gamma_+(z_3) \Gamma'_-(z_6) \Gamma_+(z_5) \cdots
    \Gamma'_-(z_{2\ell}) | \emptyset \rangle
  \end{multline}
  which can be interpreted combinatorially as a $2^\ell$-to-$1$
  correspondence between tilings of the Aztec diamonds of orders
  $\ell$ and $\ell-1$. We readily recover Stanley's formula
  \eqref{eq:stanley} by induction.
\end{rem}

\subsection{General steep tilings with prescribed boundary conditions}
\label{sec:vertexfix}

We now consider the enumeration of not necessarily pure steep tilings
with prescribed boundary conditions (i.e.\ the tiled region is fixed):
by Proposition~\ref{prop:tilseqbij} and
Lemma~\ref{lem:gfInTermsOfOperators} this amounts to evaluating the
right hand side of \eqref{eq:gfInTermsOfOperators} when
$(\alpha,\beta) \neq (\emptyset,\emptyset)$.

Let us first consider the case of plane overpartitions discussed in
Section~\ref{sec:overpart}, where $w=+^{2\ell}$, $\alpha=\emptyset$
and $\beta=\lambda$. The corresponding generating function is
\begin{equation}
  \label{eq:overpartgf}
  \begin{split}
    T_{+^{2\ell},\emptyset,\lambda} &= \langle \emptyset |
    \Gamma_+(y_1) \Gamma'_+(y_2) \Gamma_+(y_3) \Gamma'_+(y_4) 
    \cdots \Gamma_+(y_{2\ell-1}) \Gamma'_+(y_{2\ell}) | \lambda \rangle \\
    &= s_\lambda(y_1,y_3,\ldots,y_{2\ell-1}/y_2,y_4,\ldots,y_{2\ell})
  \end{split}
\end{equation}
where $s_\lambda(\cdot/\cdot)$ is a \emph{super Schur function}, also
known as hook Schur function, see \cite{Kra96} and references
therein. Indeed, the second equality of \eqref{eq:overpartgf} may be
obtained by moving all $\Gamma_+'$ to the right (since all operators
commute with each other) and observing that the resulting expression
counts some super semistandard tableaux, also called
$(\ell,\ell)$-semistandard tableaux \cite{Rem84}. Alternatively, a
direct bijection between plane overpartitions and super semistandard
tableaux was given in \cite[Remark~1]{CSV}. If we specialize $x_i=q$
for all $i$, so that $y_i=q^{2\ell+1-i}$, the generating function
specializes to
$s_{\lambda}(q^2,q^4,\ldots,q^{2\ell}/q,q^3,\ldots,q^{2\ell-1})$: to
the best of our knowledge, there is no known hook-content-type formula
for this specialization, except in the $\ell \to \infty$ limit where
we have \cite{Kra96,CSV}
\begin{equation}
  s_{\lambda}(a q^2, a q^4,\ldots/b q, b q^3,\ldots) =
  q^{2|\lambda|} \prod_{\rho \in \lambda} \frac{a + b q^{2c(\rho)-1}}{1-q^{2h(\rho)}}.
  \label{eq:sss}
\end{equation}
Here the product is over all cells $\rho$ of the Young diagram of
$\lambda$, $h(\rho)$ and $c(\rho)$ being respectively the hook length
and the content of $\rho$, and $a$ and $b$ are extra parameters that
in our context count the respective numbers of east- and south-going
dominos. (The $\ell \to \infty$ limit is well-defined if the steep
tilings of asymptotic data $+^{2\ell}$ are first translated by
$(-2\ell,2\ell)$ so that they eventually fill the whole $y \geq x$
half-plane as $\ell \to \infty$. Those tilings are nothing but
``half-pyramid partitions''.)

\begin{rem}
  If we set $a_i=c_i=q^{2i-1}$ and $b_{i}=d_{i+1}=q^{2i}$ in the
  Cauchy identity for super Schur functions \cite{Rem84}
  \begin{equation}
    \sum_\lambda s_\lambda(a/b)s_\lambda(c/d)=\prod_{i,j}\frac{(1+a_id_j)(1+b_ic_j)}{(1-a_ic_j)(1-b_id_j)},
  \end{equation}
  then we get back the generating function of pyramid partitions in
  $P_{2\ell}$. This can be related to the fact, already noted in
  Section~\ref{sec:overpart}, that pairs of plane overpartitions with
  compatible shapes are in correspondence with pyramid partitions.
\end{rem}

Let us now discuss the more general situations. For $w=+^{2\ell}$ and
arbitrary boundary conditions $\alpha$ and $\beta$, we have
\begin{equation}
  \label{eq:skewoverpartgf}
  \begin{split}
    T_{+^{2\ell},\alpha,\beta} &= \langle \alpha |
    \Gamma_+(y_1) \Gamma'_+(y_2) \Gamma_+(y_3) \Gamma'_+(y_4) 
    \cdots \Gamma_+(y_{2\ell-1}) \Gamma'_+(y_{2\ell}) | \beta \rangle \\
    &= s_{\beta/\alpha}(y_1,y_3,\ldots,y_{2\ell-1}/y_2,y_4,\ldots,y_{2\ell})
  \end{split}
\end{equation}
where $s_{\beta/\alpha}(\cdot,\cdot)$ is a skew super Schur function
(or $(\ell,\ell)$-hook skew Schur function).
Let us instead consider a general word $w=\{+,-\}^\ell$ and boundary
conditions of the form $\alpha=\emptyset$, $\beta=\lambda$. Let
$i_1<i_2<\cdots<i_n$ and $i'_1<i'_2<\cdots<i'_m$ denote the
respectively odd and even positions of the $+$'s in $w$. Upon moving
all $\Gamma_-$ and $\Gamma'_-$ to the left in
\eqref{eq:gfInTermsOfOperators} where they are ``absorbed'' by
$\langle \emptyset |$, we find that
\begin{equation}
  T_{w,\emptyset,\lambda} =  T_w s_\lambda(y_{i_1},y_{i_2},\ldots,y_{i_n}/y_{i'_1},y_{i'_2},\ldots,y_{i'_m}).
\end{equation}
Finally, for general $w$, $\alpha$ and $\beta$, we may recast
$T_{w,\alpha,\beta}/T_w$ as a (finite) sum of the form $\sum_\nu
s_{\alpha/\nu}(\cdot/\cdot) s_{\beta/\nu}(\cdot/\cdot)$. Writing down
an explicit formula is left to the interested reader. We are not aware
of any specialization of the super Schur functions besides
\eqref{eq:sss} that would provide a ``nice'' formula for
$T_{w,\alpha,\beta}$ for generic $w$ or $\alpha$.

\subsection{General steep tilings with free boundary conditions}
\label{sec:vertexfree}

We now wish to study steep tilings with ``free'' boundary conditions,
i.e.\ the tiled region is not prescribed, or in other words we do not
specify the first and last element of their corresponding sequences of
interlaced partitions. We thus consider generating functions of the
form
\begin{equation}
  \label{eq:Fwfor}
  F_w(u,v) = \sum_{\alpha,\beta} u^{|\alpha|} v^{|\beta|} T_{w,\alpha,\beta}
\end{equation}
which count \emph{all} steep tilings with asymptotic data $w$, the
exponents of $u$ and $v$ recording the numbers of boundary flips
needed on both sides to obtain a given tiling from the pure minimal
tiling $T_\mathrm{min}^w$. Interestingly, it is possible to obtain a
nice expression for $F_w(s,t)$ by the vertex operator formalism, as we
will now explain.

Introducing the \emph{free boundary states}
\begin{equation}
  \label{eq:fbsdef}
  \langle \underline{u} | = \sum_{\lambda} u^{|\lambda|} \langle \lambda |
  \qquad \text{and} \qquad
  | \underline{v} \rangle = \sum_{\lambda} v^{|\lambda|} | \lambda \rangle,
\end{equation}
where the sums range over all partitions, it immediately follows from
\eqref{eq:gfInTermsOfOperators} that
\begin{equation}
  F_w(u,v) = \langle\underline{u}|
  \prod_{i=1}^\ell \Gamma_{w_{2i-1}}(y_{2i-1}^{w_{2i-1}})
  \Gamma'_{w_{2i}}(y_{2i}^{w_{2i}}) \,
  |\underline{v}\rangle\label{eq:Fwvertexp},
\end{equation}
where we recall that $y_i=x_ix_{i+1}\cdots x_{2\ell}$.  To evaluate
this expression, it is necessary to understand how the $\Gamma$
operators act on the free boundary states.

\begin{prop}[Reflection relations]
  We have
  \begin{equation}
    \label{eq:reflrel}
    \begin{split}
      \Gamma_+(t)\, |\underline{v}\rangle &= \frac{1}{1-tv} \Gamma_-(t
      v^2)\, |\underline{v}\rangle,\\
      \Gamma'_+(t)\, |\underline{v}\rangle &= \frac{1}{1-tv} \Gamma'_-(t
      v^2)\, |\underline{v}\rangle,\\
      \langle\underline{u}|\, \Gamma_-(t) &= \frac{1}{1-tu}
      \langle\underline{u}|\, \Gamma_+(t u^2),\\
      \langle\underline{u}|\, \Gamma'_-(t) &= \frac{1}{1-tu}
      \langle\underline{u}|\, \Gamma'_+(t u^2).
    \end{split}
  \end{equation}
\end{prop}

\begin{proof}
  These amount to \cite[I.5, Ex.\ 27(a), (3)]{McDo} but let us provide
  here a combinatorial derivation.  It is sufficient to establish the
  first relation, which implies the others by conjugation and
  duality. This amounts to proving that, for any partition $\mu$, we
  have
  \begin{equation}
    \sum_{\lambda: \lambda \succ \mu} t^{|\lambda/\mu|} v^{|\lambda|} = \frac{1}{1-tv}
    \sum_{\nu: \nu \prec \mu} (tv^2)^{|\mu/\nu|} v^{|\nu|}.
  \end{equation}
  Given $\lambda$ such that $\lambda \succ \mu$, set $\nu_i=\mu_i +
  \mu_{i+1} - \lambda_{i+1}$ and $k=\lambda_1-\mu_1$: it is readily
  checked that $\nu$ is a partition such that $\nu \prec \mu$,
  satisfying $|\lambda/\mu|=|\mu/\nu|+k$, and that the mapping
  $\lambda \mapsto (\nu,k)$ is bijective ($k$ being an arbitrary
  nonnegative integer). The wanted identity follows.
\end{proof}

We now apply the reflection relations to the evaluation of
\eqref{eq:Fwvertexp}. As a warmup, let us consider the case of plane
overpartitions where $u=0$ (so that the left boundary
state is $\langle \emptyset |$) and $w=+^{2\ell}$: it is
straightforward to check that
\begin{equation}
  \label{eq:superlittlewood}
  F_{(+^{2\ell})}(0,v) = \prod_{i=1}^{2\ell} \frac{1}{1-v y_i}
  \prod_{\substack{1 \leq i < j \leq 2\ell \\ j-i \text{ odd}}} (1 +
  v^2 y_i y_j) \prod_{\substack{1 \leq i < j \leq 2\ell \\ j-i
      \text{ even}}} \frac{1}{1 - v^2 y_i y_j}
\end{equation}
(``bounce'' the $\Gamma_+/\Gamma'_+$ on $|\underline{v}\rangle$ then
move the resulting $\Gamma_-/\Gamma'_-$ to the left where they are
``absorbed'' by $\langle \emptyset |$, and collect all factors
obtained on the way).

\begin{rem}
  Recalling that $F_{(+^{2\ell})}(0,v)=\sum_\lambda v^{|\lambda|}
  s_\lambda(y_1,y_3,\ldots/y_2,y_4,\ldots)$, the
  expression~\eqref{eq:superlittlewood} is actually equivalent to the
  so-called Littlewood identity \cite[I.5, Ex.\ 4]{McDo} (recovered by
  taking $v=1$ and, say, $y_{2i-1}=z_i$ and $y_{2i}=0$ for
  $i=1,\ldots,\ell$).
\end{rem}

\begin{rem}
  We recover the generating function of plane overpartitions of
  arbitrary shape computed in \cite{CSV} by taking the appropriate
  weight specialization, namely $v=1$ and, for $i=1,\ldots,\ell$,
  $y_{2i-1}=q^{\ell+1-i}$ and $y_{2i}=a q^{\ell+1-i}$. (Note that we
  do not quite recover \cite[Theorem~6]{CSV} which contains a typo,
  but the correct formula which is on the second line of the first
  equation before \cite[Theorem~13]{CSV}.)
\end{rem}

Before writing down formulas in more general situations, let us recall
the shorthand notation
\begin{equation}
  \label{eq:phidef2}
  \varphi_{i,j}(x) =
  \begin{cases}
    1 + x & \text{if $j-i$ is odd,}\\
    1/(1-x) & \text{if $j-i$ is even.}
  \end{cases}
\end{equation}
Keeping $u=0$ (mixed boundary conditions) but taking $w$ general, we
easily obtain
\begin{equation}
  F_w(0,v) = \prod_{i:\ w_i=+} \frac{1}{1-v y_i}
  \prod_{\substack{ i < j \\ w_i=+,\ w_j=-}} \varphi_{i,j}(y_i/y_j)
  \prod_{\substack{ i < j \\  w_i=w_j=+}} \varphi_{i,j}(v^2 y_i y_j).
\end{equation}
By taking $v=1$ and $x_i=q$ for all $i$, i.e.\ $y_i=q^{2\ell+1-i}$, we
obtain the expression \eqref{eq:mainmixed} announced in
Theorem~\ref{thm:mainfree}.

\begin{figure}[htpb]
  \centering
  \includegraphics[width=.4\textwidth]{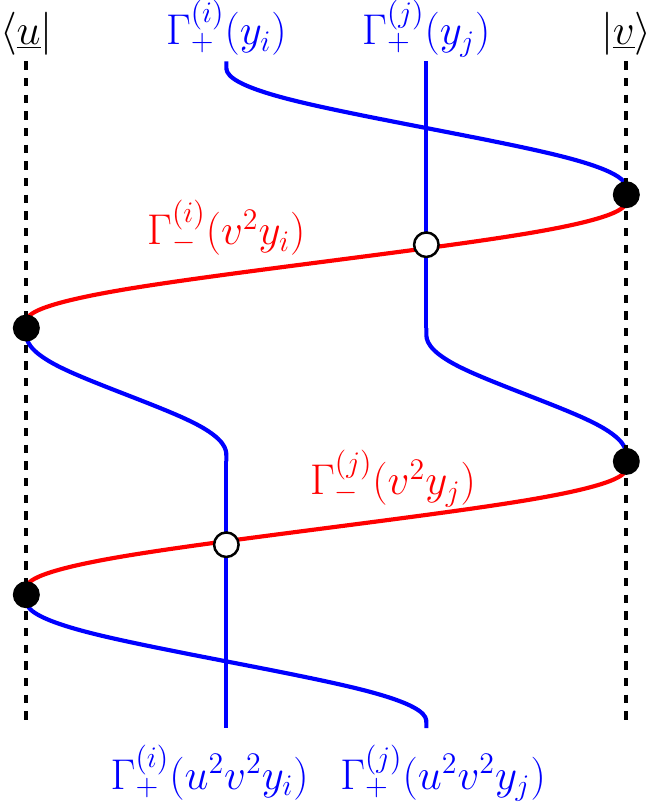}
  \caption{Schematic picture of the computation of
    \eqref{eq:Fwvertexp} in the case $w=+^{2\ell}$. We get a factor
    contributing to \eqref{eq:Fuveval} when an operator bounces on the
    boundary (black circles) or when a $\Gamma_+/\Gamma'_+$ crosses a
    $\Gamma_-/\Gamma'_-$ (white circles). Here we only represent the
    trajectories of two operators (with $\Gamma^{(k)}=\Gamma$ if $k$
    is odd and $\Gamma'$ otherwise). }
  \label{fig:freeBoundary}
\end{figure}

Slightly more involved expressions, involving infinite products, arise
when considering the case where both $u$ and $v$ are nonzero. Again we
begin with the case $w=+^{2\ell}$.  The strategy to evaluate
\eqref{eq:Fwvertexp} in this case is to pick each $\Gamma_+/\Gamma'_+$
(say, successively from left to right), move it to the right and
bounce it on $| \underline{v} \rangle$, move the resulting
$\Gamma_-/\Gamma'_-$ to the left and bounce it on $\langle
\underline{u} |$, then finally put back the resulting
$\Gamma_+/\Gamma'_+$ into place, see Figure~\ref{fig:freeBoundary} for
an illustration. In this process, we collect some factors arising from
the bounces and the crossings between operators with different
indices, and we end up with the same expression as at the beginning
except that every $\Gamma$ has its parameter multiplied by $u^2 v^2$.
In more explicit terms, we have
\begin{multline}
  \label{eq:Fuveval}
  \langle \underline{u} |\, \Gamma_+(y_1) \Gamma'_+(y_2) \cdots
  \Gamma_+(y_{2\ell-1}) \Gamma'_+(y_{2\ell})\, |\underline{v} \rangle = \\
  \prod_{i=1}^{2\ell} \frac{1}{(1-v y_i)(1-u v^2 y_i)} \prod_{1 \leq i
    < j \leq 2\ell} \varphi_{i,j}(v^2 y_i
  y_j) \varphi_{i,j}(u^2 v^4 y_i y_j) \times \\
  \langle \underline{u} |\, \Gamma_+(u^2 v^2 y_1) \Gamma'_+(u^2 v^2
  y_2) \cdots \Gamma_+(u^2 v^2 y_{2\ell-1}) \Gamma'_+(u^2 v^2
  y_{2\ell})\, |\underline{v} \rangle.
\end{multline}
Upon iterating this relation $k$ times, we pull out more factors, with
a remaining product of $\Gamma_+/\Gamma'_+$ operators with parameters
of the form $(uv)^{2k} y_i$. But we have
\begin{equation}
  \lim_{k \to \infty} \Gamma_+((uv)^{2k} y_i) =
  \lim_{k \to \infty} \Gamma'_+((uv)^{2k} y_i) = 1
\end{equation}
(where the $\Gamma$ shall be viewed as infinite matrices whose
coefficients are formal power series), hence the remaining product
tends to $\langle \underline{u} | \underline{v} \rangle = \prod_{k
  \geq 1} 1/(1-u^kv^k)$. Rearranging the factors, we end up with the
expression
\begin{equation}
    \label{eq:Fuvformplus}
  F_{(+^{2\ell})}(u,v) = \prod_{k = 1}^{\infty} \left(
    \frac{1}{1-u^k v^k} \prod_{i=1}^{2\ell}
    \frac{1}{1-u^{k-1} v^k y_i} \prod_{1 \leq i
    < j \leq 2\ell} \varphi_{i,j}(u^{2k-2} v^{2k} y_i y_j) \right).
\end{equation}
Finally, for general $w$, a straightforward adaptation of our strategy
yields
\begin{multline}
  \label{eq:Fuvformgen}
  F_w(u,v) = \prod_{k = 1}^{\infty} \Bigg( \frac{1}{1-u^k v^k}
    \prod_{i:\, w_i=+} \frac{1}{1-u^{k-1} v^k y_i}
    \prod_{i:\, w_i=-} \frac{1}{1-u^k v^{k-1}/y_i} \\
    \prod_{\substack{1 \leq i < j \leq 2\ell \\ w_i=+,\ w_j=-}}
    \varphi_{i,j}(u^{2k-2} v^{2k-2} y_i/y_j)
    \prod_{\substack{1 \leq i < j \leq 2\ell \\ w_i=w_j=+}}
    \varphi_{i,j}(u^{2k-2} v^{2k} y_i y_j)\\
    \prod_{\substack{1 \leq i < j \leq 2\ell \\ w_i=-,\ w_j=+}}
    \varphi_{i,j}(u^{2k} v^{2k} y_j/y_i)
    \prod_{\substack{1 \leq i < j \leq 2\ell \\ w_i=w_j=-}}
    \varphi_{i,j}(u^{2k} v^{2k-2}/(y_i y_j)) \Bigg).
\end{multline}
The reader might be wary of the divisions by $y_i$ or $y_j$, but
recall from Lemma~\ref{lem:gfInTermsOfOperators} that
$y_i=x_ix_{i+1}\cdots x_{2\ell}$ where $x_i$ records the size
difference between $\lambda^{(i)}$ and $\alpha=\lambda^{(0)}$.  To
obtain a \emph{bona fide} power series we need to do the change of
variables $u \to u y_0$, and we may then even set $u=v=1$ without
trouble to obtain the ``true'' generating function of all steep
tilings counted with a weight $x_i$ per flip on the $i$-th diagonal.
In particular, if we take $x_i=q$ for all $i$, hence
$y_i=q^{2\ell+1-i}$, $u=q^{2\ell+1}$ and $v=1$ in
\eqref{eq:Fuvformgen}, then by rearranging the products we obtain the
expression \eqref{eq:mainfree} announced in
Theorem~\ref{thm:mainfree}, whose proof is now complete.

\begin{rem}
  The other equations in \cite[I.5, Ex.\ 27]{McDo} suggest other types
  of ``free'' boundary conditions. For instance, if we consider the
  state
  \begin{equation}
    | \underline{\tilde{v}} \rangle = \sum_{\lambda \text{ even}}
    v^{|\lambda|} | \lambda \rangle = \Gamma_-(-v) |\underline{v} \rangle
  \end{equation}
  (where a partition is said even if all its parts are even), then we
  have the modified reflection relations
  \begin{equation}
    \Gamma_+(t)\, | \underline{\tilde{v}} \rangle = 
    \frac{1}{1-t^2 v^2} \Gamma_-(t v^2)\, | \underline{\tilde{v}} \rangle,
    \qquad
    \Gamma'_+(t)\, | \underline{\tilde{v}} \rangle = 
    \Gamma'_-(t v^2)\, | \underline{\tilde{v}} \rangle
  \end{equation}
  and we may then obtain different product formulas for steep tilings
  with such boundary condition. Considering instead a sum over
  partitions of the form
  $(\alpha_1-1,\ldots,\alpha_p-1|\alpha_1,\ldots,\alpha_p)$ in
  Frobenius notation, we obtain a boundary state that mutates a
  $\Gamma_+$ into a $\Gamma'_-$ and a $\Gamma'_+$ into a $\Gamma_+$,
  up to factors (it would be interesting to have a combinatorial proof
  of this fact).
\end{rem}

\begin{rem}
  For a random steep tiling with mixed boundary conditions, the
  associated sequence of interlaced partitions forms a so-called
  \emph{Pfaffian Schur process} \cite{BR05}. The case with two free
  boundaries has, to the best of our knowledge, not been considered
  before.
\end{rem}

\section{Cylindric steep tilings}
\label{sec:cyl}

By a variant of our approach we may consider \emph{cylindric steep
  tilings} of width $2\ell$: these may be viewed as domino tilings of
the plane which are periodic in one direction, namely they are
invariant under a translation of vector $(c,c-2\ell)$ for some $c$,
and which are steep in the same sense as before, that is we only find
north- or east-going (resp.\ south- or west-going) dominos
sufficiently far away in the north-east (resp.\ south-west)
direction. We may restrict a cylindric steep tiling to a fundamental
domain by cutting it ``along'' the $y=x$ and $y=x-2\ell$ lines (more
precisely we cut along the lattice paths that remain closest to these
lines and follow domino boundaries): we then obtain a steep tiling of
the oblique strip as before, with the only additional constraint that
the two boundaries must ``fit'' into each other. We define the
asymptotic data $w \in \{+,-\}^{2\ell}$ as before.

Then, we may proceed as in Section~\ref{sec:bijection} and construct
the particle configuration and sequence of integer partitions
associated to the tiling. Clearly, the additional constraint is that
the particle configuration on the lines $y=x$ and $y=x-2\ell$ must be
the same up to translation, which implies that the associated
partitions $\lambda^{(0)}$ and $\lambda^{(2\ell)}$ are equal, and that
the parameter $c$ above must be equal to $c_{2\ell}-c_0$, as defined
by \eqref{eq:xxplimrel} (the cylindric steep tiling is said centered
if $c_0=0$). We readily arrive at an analogue of
Proposition~\ref{prop:tilseqbij}, with however a caveat regarding
flips. Indeed, when viewing a cylindric steep tiling as a periodic
tiling, a flip consists in rotating a $2 \times 2$ block of dominos
and all its translates (as we wish to preserve periodicity). When
considering the tiling restricted to a fundamental domain, regular
(bulk) flips are defined as before, but a boundary flip is only
allowed if a corresponding flip can be and is performed on the other
boundary in order to preserve the shape compatibility (this pair of
moves is a flip centered on the $2\ell$-th diagonal).

\begin{prop}
  \label{prop:cyltilseqbij}
  Given a word $w \in \{ +,- \}^{2\ell}$, there is a bijection between
  the set of centered cylindric steep tilings with asymptotic data $w$
  and the set of sequences of partitions
  $(\lambda^{(0)},\ldots,\lambda^{(2\ell)})$ with
  $\lambda^{(0)}=\lambda^{(2\ell)}$ and such that, for all $k \in
  \{1,\ldots,\ell\}$,
  \begin{itemize}
  \item $\lambda^{(2k-2)} \prec \lambda^{(2k-1)}$ if $w_{2k-1}=+$, and
    $\lambda^{(2k-2)} \succ \lambda^{(2k-1)}$ if $w_{2k-1}=-$,
  \item $\lambda^{(2k-1)} \prec' \lambda^{(2k)}$ if $w_{2k}=+$, and
    $\lambda^{(2k-1)} \succ' \lambda^{(2k)}$ if $w_{2k}=-$.
  \end{itemize}
  Furthermore, the bijection has the following properties.
  \begin{itemize}
  \item[B'.] For $m=1,\ldots,2\ell$, the absolute value of
    $|\lambda^{(m)}|-|\lambda^{(m-1)}|$ counts the number of dominos
    whose centers are on the line $y=x-m+1/2$ and whose orientations
    are opposite to the asymptotic one, as detailed in
    Table~\ref{tab:orient}.
  \item[C'.] If $w$ contains at least one $+$ and one $-$, then for
    any $m=1,\ldots,2\ell$, $|\lambda^{(m)}|$ counts the number of
    flips centered on the $m$-th diagonal in any minimal sequence of
    flips between the tiling at hand and the minimal tiling
    corresponding to the sequence
    $(\emptyset,\emptyset,\ldots,\emptyset)$.
  \end{itemize}
\end{prop}

\begin{proof}
  The bijectivity and the property B' are obtained along the same
  lines as for Proposition~\ref{prop:tilseqbij}. We only detail the
  proof of the property C' since it involves a slight subtlety. We
  need again to consider height functions, which typically become
  ``multivalued'' functions in the cylindric setting.  Indeed, using
  \eqref{eq:Hasymp} and the characterizations of $c_m$ and $h_m$ in
  terms of the asymptotic data $w$, we find that, for a periodic steep
  tiling with period $(c,c-2\ell)$ and any $(x,y)$ we have
  \begin{equation}
    \label{eq:Hquasiper}
    H(x,y) - H(x+c,y+c-2\ell) = h,
  \end{equation}
  where $h=h_0-h_{2\ell}$ is equal to the number of $+$ in $w$ minus
  that of $-$. For $h \neq 0$, $H$ is quasiperiodic but not periodic
  in the plane hence multivalued on the cylinder. Since the minimal
  height function $H_{\mathrm{min}}^w$ has the same quasiperiodicity
  property, the ``reduced'' height function
  $\widetilde{H}=(H-H_{\mathrm{min}}^w)/4$ is periodic. To a
  cylindring steep tiling $T$ with height function $H$, we associate
  its rank given by
  \begin{equation}
    \label{eq:rtcyl}
    r_w(T) = \sum_{(x,y)\in \widetilde{V}} \widetilde{H}(x,y) = \sum_{m=0}^{2\ell-1} |\lambda^{(m)}|
  \end{equation}
  where
  $\widetilde{V} = \{ (x,y) \in \mathbb{Z}^2, 0 \leq x-y < 2\ell \}$. This
  definition is the cylindric analogue of \eqref{eq:rt} and
  \eqref{eq:dfliptmin}. The property C' results from the following
  claim: \emph{if $r_w(T)>0$ and $w$ is neither $+^{2\ell}$ nor
    $-^{2\ell}$, then $T$ admits a descendent flip} (reducing the rank
  by $1$).

  Such a flip can be obtained by the criterion given in the erratum
  of \cite{EKLP1992}: we first consider the vertices where
  $\widetilde{H}$ is maximal, and we look for one among them where $H$
  is (locally) maximal. If such a vertex exists, it is not difficult
  to see that we may perform a descendent flip on it. But, as $H$ is
  unbounded in the plane for $h\neq 0$, the existence of such a vertex
  is a priori not obvious. Observe first that, since $\widetilde{H}$
  is periodic and $0<r_w(T)<\infty$, $\widetilde{H}$ attains its
  maximal value on a nonempty subset $S$ of $\mathbb{Z}^2$ which
  intersects the fundamental domain $\widetilde{V}$ at finitely many
  points. By contraposition, proving our claim boils down to showing
  that, if $H$ admits no local maximum within $S$, then $w$ is
  necessarily either $+^{2\ell}$ or $-^{2\ell}$.

  Assuming that $H$ admits no local maximum within $S$, we may
  construct an infinite walk $v_0,v_1,\ldots$ in $\mathbb{Z}^2$ as
  follows: $v_0$ is an arbitrary element of $S$ and, assuming that
  $v_{i-1}$ has been constructed, we pick $v_i$ as one of its nearest
  neighbors in $\mathbb{Z}^2$ such that $H(v_i)>H(v_{i-1})$. It is
  easily seen that $v_i\in S$ for all $i$, and since
  $S \cap \widetilde{V}$ is finite, the projection of the walk on the
  cylinder eventually intersects itself, in other words there exists
  $j<j'$ such that $v_{j'}-v_j=k(c,c-2\ell)$ for some
  $k \in \mathbb{Z}$. Note that
  $0 \neq H(v_{j'})-H(v_j) \geq j'-j \geq 2\ell|k|$ by the
  construction of our lattice walk. But we have $H(v_{j'})-H(v_j)=-kh$
  from \eqref{eq:Hquasiper}, so we conclude that $|h| \geq 2\ell$
  hence $w=+^{2\ell}$ or $-^{2\ell}$ as wanted.
\end{proof}

\begin{rem}
  In the case $w=+^{2\ell}$ or $-^{2\ell}$, then clearly any sequence
  of interlaced partitions satisfying the conditions of
  Proposition~\ref{prop:cyltilseqbij} is constant. The rank $r_w(T)$
  of the corresponding tiling $T$ is a multiple of $2\ell$, hence
  there is a clear obstruction to the property C' and to the presence
  of flips in $T$. This is actually not due to the cylindrical
  topology \emph{per se}, but to the presence of ``forced cycles'' in
  the associated perfect matching, see \cite[Example~2.3]{Propp1993}
  (in fact Jockusch's graph is isomorphic to a truncated cylindric
  tilted square grid of circumference 4, thus corresponds essentially
  to the same situation as ours).
\end{rem}

The enumerative consequence of Proposition~\ref{prop:cyltilseqbij} is
the following statement, which readily implies Theorem~\ref{thm:maincyl}
by taking $x_i=q$ for all $i$.

\begin{thm}
  \label{thm:cylgf}
  Let $w \in \{+,-\}^{2\ell}$ be a word. Let $C_w \equiv
  C_w(x_1,\ldots,x_{2\ell})$ be the generating function of
  cylindric steep tilings of asymptotic data $w$, where the exponent
  of the variable $x_i$ records the number of flips centered on the
  $i$-th diagonal in a shortest sequence of flips from the minimal
  tiling. Then one has
  \begin{multline}
    \label{eq:cylgf}
    C_w = \prod_{k \geq 1} \Bigg( \frac{1}{1-y^k}
    \prod_{\substack{1 \leq i < j \leq 2\ell \\ w_i=+,\ w_j=-}}
    \varphi_{i,j}(y^{k-1} x_i x_{i+1} \cdots x_{j-1}) \\
    \prod_{\substack{1 \leq i < j \leq 2\ell \\ w_i=-,\ w_j=+}}
    \varphi_{i,j}(y^{k-1} x_1 x_2 \cdots x_{i-1} x_j x_{j+1} \cdots
    x_{2\ell}) \Bigg)
  \end{multline}
  where $y=x_1 \cdots x_{2\ell}$ and $\varphi_{i,j}(\cdot)$ is defined
  as in \eqref{eq:phidef}.
\end{thm}

\begin{rem}
  Proposition~\ref{prop:cyltilseqbij} actually shows that cylindric
  steep tilings (weighted by flips) form a \emph{periodic Schur
    process}, as defined by Borodin \cite{B2007}, and \eqref{eq:cylgf}
  is nothing but the partition function of this process.
\end{rem}

\begin{proof}
  We could obtain \eqref{eq:cylgf} as a suitable specialization of
  \cite[Proposition~1.1]{B2007}, itself a variation on \cite[I.5, Ex.\
  28(a)]{McDo}, but let us here provide a proof using vertex
  operators. By Proposition~\ref{prop:cyltilseqbij}, we have
  \begin{equation}
    C_w = \sum_\lambda y^{|\lambda|} T_{w,\lambda,\lambda}
  \end{equation}
  where $T_{w,\lambda,\lambda}$ is defined as in
  Lemma~\ref{lem:gfInTermsOfOperators}. We evaluate this quantity
  following a strategy similar to that used for the derivation of
  \eqref{eq:Fuvformplus} and \eqref{eq:Fuvformgen} in the case of
  steep tilings with free boundary conditions, but here we would like
  the cyclic symmetry to be manifest, which leads us to rewrite
  $C_w$ in a slightly different form. Let us introduce the
  \emph{energy operator} $H$ such that $H | \lambda \rangle =
  |\lambda| \times | \lambda \rangle$, so that $x^H | \lambda \rangle
  = x^{|\lambda|} | \lambda \rangle$ with $x$ a formal variable. Then,
  we have
  \begin{equation}
    \label{eq:twctr}
    C_w = \mathrm{Tr} \left[ \Gamma_{w_1}(1) (x_1)^H \Gamma'_{w_2}(1) (x_2)^H \cdots  \Gamma_{w_{2\ell-1}}(1) (x_{2\ell-1})^H \Gamma'_{w_{2\ell}}(1) (x_{2\ell})^H\right]
  \end{equation}
  where $\mathrm{Tr}[\cdot] = \sum_\lambda \langle \lambda | \cdot |
  \lambda \rangle$ is the trace. The strategy is then to move, one by
  one, each $\Gamma_+$ or $\Gamma'_+$ to the right and ``wrap'' it
  around the cylinder (using the cyclicity of the trace) until it is
  back into place. To that end, we need the following easy commutation
  relations
  \begin{equation}
    \Gamma_+(t)\, x^H = x^H \Gamma_+(t x), \qquad
    \Gamma'_+(t)\, x^H = x^H \Gamma'_+(t x),
  \end{equation}
  and the usual commutation relations \eqref{eq:gamcom}, which show
  that a factor $\varphi_{i,j}$ arises each time a
  $\Gamma_+/\Gamma'_+$ crosses a $\Gamma_-/\Gamma'_-$, with an
  argument equal to the product of the $x_k$'s between them, and
  furthermore that the argument of each $\Gamma_+/\Gamma'_+$ is
  multiplied by $y$ after one turn. In other words we have
  \begin{multline}
    C_w =  \prod_{\substack{1 \leq i < j \leq 2\ell \\ w_i=+,\ w_j=-}}
    \varphi_{i,j}(x_i x_{i+1} \cdots x_{j-1}) \times \\
    \prod_{\substack{1 \leq i < j \leq 2\ell \\ w_i=-,\ w_j=+}}
    \varphi_{i,j}(x_1 x_2 \cdots x_{i-1} x_j x_{j+1} \cdots
    x_{2\ell}) \times \widetilde{C}_w(y)
  \end{multline}
  where $\widetilde{C}_w(y)$ is obtained from the right hand side of
  \eqref{eq:twctr} by replacing the argument of each
  $\Gamma_+/\Gamma'_+$ by $y$. We may then repeat the same strategy
  $k$ times, pulling more factors times $\widetilde{C}_w(y^k)$. But
  then it is readily seen that
  \begin{equation}
    \lim_{k \to \infty} \widetilde{C}_w(y^k) =
    \mathrm{Tr}\left[ (x_1 \cdots x_{2\ell})^H \right] =
    \prod_{k \geq 1} \frac{1}{1-y^k}
  \end{equation}
  from the (already noted) fact that $\lim \Gamma_+(y^k) = \lim
  \Gamma'_+(y^k) = 1$ as $k \to \infty$ and that the
  $\Gamma_-/\Gamma'_-$ are ``upper unitriangular''. The wanted
  expression \eqref{eq:cylgf} follows.
\end{proof}

\section{An extended model: interpolation between plane partitions and domino tilings}\label{sec:extendedmodel}

In this section we define an extended model that is more general than steep tilings. The model gives an interpretation of any sequence of partitions interlaced with relations in $\{\prec, \succ, \prec', \succ'\}$ in terms of 
perfect matchings of some infinite planar graph (which also gives an
interpretation in terms of tilings, see Section~\ref{sec:extendedtilings}). The model contains both steep tilings and plane partitions as special cases. For simplicity we only deal with the case of pure boundary conditions, but there is no doubt that we may also treat more general (mixed, free, periodic) ones.

Let $k \geq 1$, and let 
$\diamond\in\{\succ,\prec,\succ',\prec'\}^{k}$ be a word.
We consider the set $\mathcal{S}_\diamond$ made by sequences of partitions interlaced according  to $\diamond$:
\begin{equation}
  \mathcal{S}_\diamond := \big\{ (\lambda^{(0)}, \lambda^{(1)}, \dots, \lambda^{(k)}), \text{ for all } 1 \leq i \leq k, \lambda^{(i-1)}\diamond_i \lambda^{(i)}\big\}.
\end{equation}
Proposition~\ref{prop:tilseqbij} shows that if $k$ is even, and if $\diamond_i\in \{\prec, \succ\}$ (resp.\ $\diamond_i\in \{\prec', \succ'\}$) when $i$ is odd (resp.\ even), elements of $\mathcal{S}_\diamond$ are in bijection with steep tilings of a given asymptotic data.
We will generalize the construction to words $\diamond$ that do not satisfy this condition.

In order to present the construction, it is convenient to make a step from the
world of domino tilings to the world of matchings. Recall that a
\emph{matching} of a graph is a subset of disjoint edges. A matching is
\emph{perfect} if it covers all the vertices. It is well-known (and clear) that
domino tilings of a region of $\mathbb{Z}^2$ made by a union of unit squares
are in bijection with perfect matchings of its dual graph.

The flip operation has a natural description in terms of matchings, that can be generalized as follows. Let $G$ be a bipartite plane graph, let $\mathfrak{t}$ be a matching of $G$, and let $f$ be a bounded face of $G$ bordered by $2p$ edges, for some $p\geq 1$. If $p$ of these edges belong to the matching  $\mathfrak{t}$, we can remove these edges from $\mathfrak{t}$ and replace them by the other $p$ edges bordering $f$, thus creating a new matching $\mathfrak{t}'$ of $G$. This operation is called \emph{a flip}, see Figure~\ref{fig:matchingflip}. 

\begin{figure}[htpb]
  \centering
  \includegraphics[width=\textwidth]{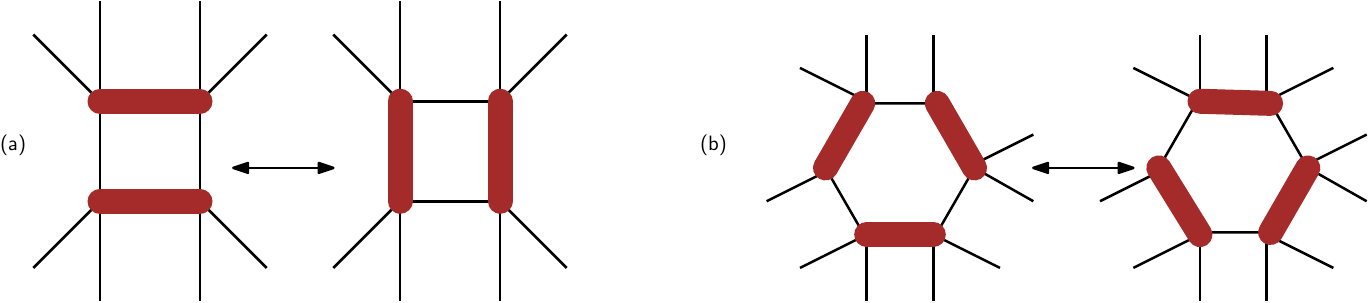}
 \caption{The flip operation around a face; (a) face of degree 4; (b) face of degree 6. Edges belonging to the matching are represented as bold.}
  \label{fig:matchingflip}
\end{figure}

\medskip

We now state the main result of this section. An element of $\mathcal{S}_\diamond$ is \emph{pure} if it is such that $\lambda^{(0)}=\lambda^{(k)}=\emptyset$.
\begin{thm}\label{thm:extendedBij}
Let $k\geq 1$ and $\diamond \in \{\prec, \succ, \prec', \succ'\}^k$. There exists an infinite plane graph ${G}_\diamond$, and a matching $\mathfrak{m}_\diamond$ of ${G}_\diamond$, called \emph{minimal}, such that pure elements of $\mathcal{S}_\diamond$ are in bijection with matchings of $G_\diamond$ that can be obtained from $\mathfrak{m}_\diamond$ by a finite sequence of flips.

The bijection maps a sequence $\vec{\lambda}=(\lambda^{(0)}, \lambda^{(1)}, \dots, \lambda^{(k)})$ to a matching $\phi(\vec{\lambda})$ such that the the minimum number of flips needed to obtain $\phi(\vec{\lambda})$ from $\mathfrak{m}_\diamond$ is equal to $\sum_{i=0}^{k} |\lambda^{(i)}|$.
\end{thm}
\begin{rem}
We will actually see a more precise result. Namely, the graph ${G}_\diamond$ and its embedding in the plane will be such that each flip has a well defined \emph{abscissa}, which will be of the form $\frac{3}{2}j$ for some $j\in\{0,1,\dots,k\}$.
Then the following will be true: the number of flips at abscissa $\frac{3}{2}j$ in any shortest sequence of flips from $\mathfrak{m}_\diamond$ to $\phi(\vec\lambda)$ is independent of the sequence, and is equal to $|\lambda^{(j)}|$. This will lead us to analogues of Theorem~\ref{thm:mainWithStanleyWeights} in the general setting (see Theorem~\ref{thm:extendedCountingWithStanleyWeights} below).
\end{rem}

\subsection{The graph $G_\diamond$, and admissible matchings.}
\label{sec:admmatch}

We now start the proof of Theorem~\ref{thm:extendedBij}.  We
first construct a graph $G_\diamond$, which is a bipartite graph embedded in
the plane (Figure~\ref{fig:extendedgraph}). The vertex set $V$ of $G_\diamond$ is defined by:
\begin{equation}
  V = \bigcup_{i=0}^k V_j \cup \bigcup_{i=1}^{k} W_j,
\end{equation}
where for $0\leq j \leq k$, 
$V_j = \{(\frac{3}{2}j,y),y\in\mathbb{Z}+\frac{j+1}{2}\}$, 
and where for $1\leq j \leq k$
\begin{equation}
  W_j =
  \begin{cases}
    \{(\frac{3}{2}j-1,y),y\in\mathbb{Z}+\frac{j+1}{2}\},& \text{if } \diamond_j\in\{\prec, \succ\}, \\
    \{(\frac{3}{2}j-\frac{1}{2},y),y\in\mathbb{Z}+\frac{j}{2}\},& \text{if } \diamond_j\in\{\prec', \succ'\}.
  \end{cases}
\end{equation}
We then add an edge of $G_\diamond$ between any two vertices of $V$ which
differ by a vector $(1,0)$, $(\frac{1}{2}, \frac{1}{2})$, or $(\frac{1}{2},-\frac{1}{2})$. See Figure~\ref{fig:extendedgraph}.
Note that the graph $G_\diamond$ does not characterize the word $\diamond$: indeed the symbols $\prec, \succ$ ($\prec', \succ'$, respectively) play the same role in the construction.

\medskip

In view of defining the  matchings we are interested in, we first need to
define a function $Y$ that will play the role a zero ordinate, local to each
``column'' of $G_\diamond$.
To this end, define $x_0=0$, and for $1\leq j \leq k$ let $x_{2k-1}$ (resp.
$x_{2k}$) be the common abscissa of all vertices in $W_{k}$ (resp.
$V_k$). Then $x_0< x_1<\dots< x_{2k}$ are all the abscissas of vertices appearing in $G_\diamond$. 
We define the  function $Y: \{x_0,x_1,\dots,x_{2k}\} \rightarrow
\frac{1}{2}\mathbb{Z}$ by the fact that $Y(x_0)=0$ and for $1\leq i \leq k$:
\begin{eqnarray}
Y(x_{2i})=\begin{cases}
    Y(x_{2i-2})+\frac12& \text{if } \diamond_j\in\{\prec, \succ'\}, \\
    Y(x_{2i-2})-\frac12& \text{if } \diamond_j\in\{\prec', \succ\}, 
  \end{cases}\label{eq:defY1}
\end{eqnarray}
\begin{eqnarray}
Y(x_{2i-1})=\begin{cases}
    Y(x_{2i})& \text{if } \diamond_j\in\{\prec, \succ\}, \\
    Y(x_{2i-2})& \text{if } \diamond_j\in\{\prec', \succ'\}.
  \end{cases}\label{eq:defY2}
\end{eqnarray}
Strictly speaking, only the values $Y(x_{2i})$ are needed in our construction, but defining $Y$ on all $x_i$'s
enables one to represent it easily on pictures, see Figure~\ref{fig:extendedgraph}.

\begin{figure}[htpb]
  \centering
  \includegraphics[scale=0.7]{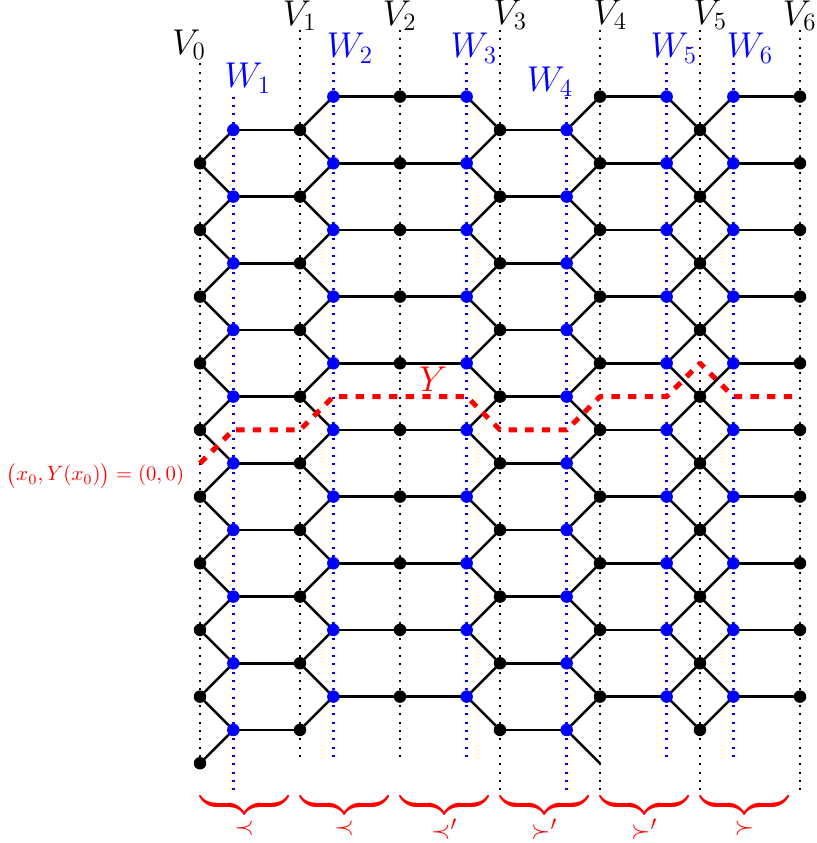}
 \caption{The graph $G_\diamond$ for $\diamond =
\prec\prec\prec'\succ'\succ'\succ$ (and $k=6$).
Only a bounded portion is shown, the actual graph being infinite towards top
and bottom. Vertical dotted lines are there to help visualize the vertex
sets $V_0,W_1,\dots, V_k,W_k$,  and they are not part of the graph. The red dotted path
is not part of the graph either, and represents the path $(x_j,Y(x_j))_{0\leq
j\leq 2k}$. In particular its leftmost point is the origin $(0,0)$
of the coordinate system. } 
  \label{fig:extendedgraph}
\end{figure}

\medskip
All the matchings of $G_\diamond$ that we will consider are such that all the
vertices of $V\setminus (V_0 \cup V_{k})$ are covered by the matching. If $v$
is a vertex of $V$, we will then say that $v$ is matched \emph{to the left}
(resp.\ \emph{to the right}), if either $v$ is covered and the matching connects
it to a vertex to its left (resp.\ to its right), or if $v$ is uncovered and
belongs to $V_0$ (resp.\ to $V_k$). Note that this definition is consistent if one
imagines that uncovered vertices are matched towards the exterior of the graph.
A matching of $G_\diamond$ is \emph{admissible} if it is such that all for each
$0\leq j\leq k$, the number of vertices in $V_j$ above the ordinate $Y(x_{2j})$
that are matched to the left is finite, and if this number equals the number of
vertices in $V_j$ below the ordinate $Y(x_{2j})$ that are matched to the right.
See Figure~\ref{fig:extendedmatching}.

\begin{figure}[htpb]
  \centering
  \includegraphics[scale=0.7]{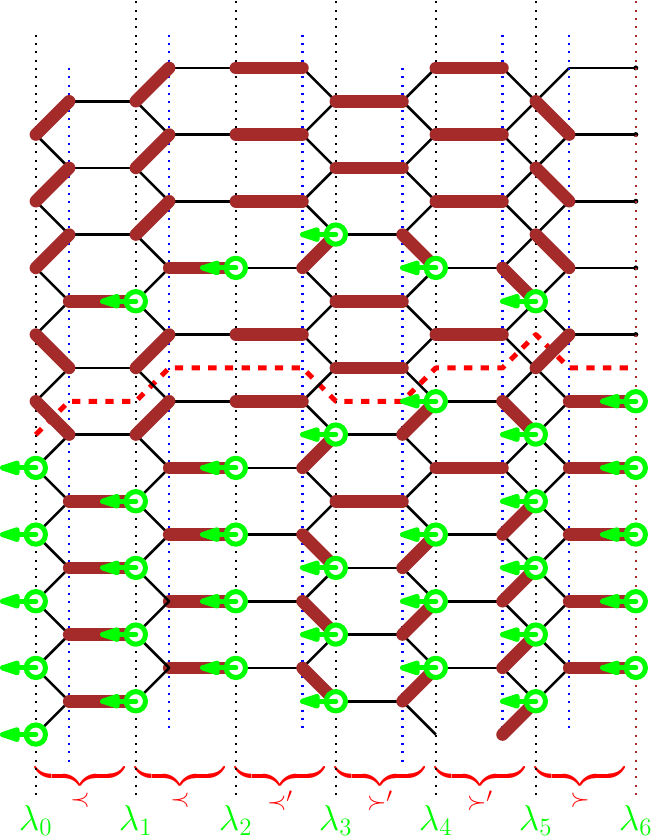}
 \caption{An admissible matching of the graph of Figure~\ref{fig:extendedgraph}. Edges in the matching are bold (and
brown). To help visualize, vertices in $V_0\cup V_1 \cup \dots \cup V_k$ that are
matched to the left are indicated by circled (green) left-pointing arrows.
Outside the displayed region, the matching continues periodically along the $y$
direction towards top and bottom. For each $j\in\{0,1,\dots,k\}$, the
bijection $\Psi$ consists in interpreting the circled vertices at abscissa $\frac{3}{2}j$ as the Maya diagram
of a partition. This example corresponds to the partitions
$\lambda^{(0)}=\emptyset$, 
$\lambda^{(1)}=\lambda^{(2)}=(2)$, 
$\lambda^{(3)}=(3,1)$, 
$\lambda^{(4)}=(2,1)$, 
$\lambda^{(5)}=(1)$, 
$\lambda^{(6)}=\emptyset$.
}
  \label{fig:extendedmatching}
\end{figure}

\medskip
Let $\mathfrak{m}$ be an admissible matching of $G_\diamond$, and let $j\in\{0,1,\dots,k\}$. Consider all the vertices of $V_j$ that are matched to the left, from top to bottom, and let $y_{1,j}> y_{2,j}> \dots$ be their ordinates.
Then the admissibility condition ensures that the nonincreasing integer sequence 
\begin{eqnarray}\label{eq:extendedinterlacing}
\lambda^{(j)}_i:=y_{i,j}-Y(x_{2j})+i-\frac{1}{2}
\end{eqnarray}
vanishes for $i$ large enough, i.e.\ that $\lambda^{(j)}$ is an integer
partition. One can interpret this construction by noting that if one considers
vertices in $V_j$ that are matched to the left (resp.\ to the right) as occupied
sites (resp.\ empty sites), then $V_j$, when read from bottom to top, is the
Maya diagram of the partition $\lambda^{(j)}$, see
Figure~\ref{fig:extendedmatching}.
 We let $\Psi(\mathfrak{m}):=(\lambda^{(0)}, \lambda^{(1)}, \dots, \lambda^{(k)})$  be the tuple of partitions thus defined.
\begin{prop}\label{prop:extendedadmissible}
The mapping $\Psi$ is a bijection between admissible matchings of $G_\diamond$
and elements of $\mathcal{S}_\diamond$.
\end{prop}

\medskip

\begin{proof}
Let us first check that $\Psi(\mathfrak{m})$ belongs to
$\mathcal{S}_\diamond$. 
 Let $j\in\{1,2,\dots,k\}$, and suppose first that
$\diamond_j\in\{\prec,\succ\}$. We claim that for $i\geq 1$, there is exactly
one vertex of $V_{j}$ whose ordinate lies between
$y_{i,j-1}$ and $y_{i+1,j-1}$. Indeed, by construction, there are
$(y_{i,j-1}-y_{i+1,j-1})$ vertices in $W_{j}$ whose ordinates are in this interval,
and exactly $(y_{i,j-1}-y_{i+1,j-1} -1)$ matching edges are coming to these
vertices from $V_{j-1}$.
Thus exactly one of these vertices is matched to the right. Since a vertex in
$W_j$ is matched to the right if and only if its unique neighbour in $V_j$ is
matched to the left, this  proves the
claim.
Equivalently, there is a unique $i'\geq 1$
such that $y_{i,j-1}>y_{i',j}>y_{i+1,j-1}$, which shows that the partitions
$\lambda^{(j-1)}$ and $\lambda^{(j)}$ are interlaced, i.e.\ either
$\lambda^{(j-1)}\prec \lambda^{(j)}$ or $\lambda^{(j-1)}\succ\lambda^{(j)}$. To
determine which case we are in, take $i$ large enough such that
$\lambda^{(j-1)}_i=\lambda^{(j)}_i=0$. Then~\eqref{eq:extendedinterlacing} and
the definition of $Y$ show that $y_{i,j-1}-y_{i,j}$ is positive
(resp.\ negative) if $\diamond_j=\ \succ$ (resp.\ $\diamond_j=\ \prec$).
This proves that $\lambda^{(j-1)}\diamond_j\lambda^{(j)}$ in both cases.
The case where $\diamond_j\in\{\prec',\succ'\}$ is handled exactly in the same
way, and we leave the reader check that
$\lambda^{(j-1)}\diamond_j\lambda^{(j)}$ in this case as well, thus
proving that $(\lambda^{(0)}, \lambda^{(1)}, \dots, \lambda^{(k)})$
belongs to $\mathcal{S}_\diamond$.

Now let $\vec\lambda=(\lambda^{(0)}, \lambda^{(1)}, \dots, \lambda^{(k)}) \in
\mathcal{S}_\diamond$. We will prove that there exists a unique admissible
matching of $G_\diamond$ such that for each $j$, the elements of $V_j$ that are
matched to the left are exactly the ones of ordinate
$\lambda^{(j)}_i+Y(x_{2j})-i+\frac{1}{2}$ for some $i\geq 1$.
First note that imposing which elements of $V_j$ are matched to the left for
all $j\in\{0,1,\dots,k\}$ determines uniquely the matching on horizontal edges
(indeed any horizontal edge is incident to such a vertex). One thus has to show
that, once these horizontal edges have been placed, there is a unique way of
adding diagonal edges to the matching in order to
make it admissible and respect the condition on left-matching vertices. This
can be easily checked, case by case, distinguishing as before according to the
value of $\diamond_j\in\{\prec, \succ, \prec', \succ'\}$. We leave it to the
reader.
\end{proof}

\subsection{The pure case, and flips}
\label{sec:extpure}
In the general case where $\lambda^{(0)}$ and $\lambda^{(k)}$ are arbitrary, we
will not go further than Proposition~\ref{prop:extendedadmissible}. We now focus on
the pure case, i.e.\ we assume $\lambda^{(0)}=\lambda^{(k)}=\emptyset$.
First, we define \emph{the minimal matching} $\mathfrak{m}_\diamond$ as the
image of $(\emptyset, \emptyset, \dots, \emptyset)$ by the bijection $\Psi^{-1}$.
See Figure~\ref{fig:extendedminimal}. 

In order to prove Theorem~\ref{thm:extendedBij}, we first  note that if $F$ is a bounded face of $G_\diamond$, there is exactly one $j\in \llbracket1,k\rrbracket$ such that two vertices of $V_j$ belong to $F$. Moreover, if the flip of the face $F$ is possible, these two vertices are matched in different directions (one to the left, the other one to the right), and the flip exchanges these directions between the two vertices. It follows that the effect of a flip on the Maya diagram is to make exactly one particle jump by one position. We define a flip to be \emph{ascendent} or \emph{descendent}  according to whether the particle jumps to the top or to the bottom, respectively. Moreover, we define  of the \emph{abscissa} of the flip to be the abscissa of this particle, i.e.\ $\frac{3}{2}j$.
Since the jump of one particle to the top increases the quantity $\sum_{i=0}^k |\lambda^{(i)}|$ by exactly one, Theorem~\ref{thm:extendedBij} is thus a direct consequence of the following lemma:
\begin{lem}
Let $\mathfrak{m}$ be a pure admissible matching which is different from the minimal one. Then it is possible to perform a descendent flip on $\mathfrak{m}$.
\end{lem}
\begin{proof}
This is a consequence of the general theory developed by Propp \cite{Propp1993}, but let us here provide a self-contained argument. 
First to each admissible matching we associate a height function that associates to each bounded face $F$ of the graph $G_\diamond$ a value $h(F)\in \mathbb{Z}$ as follows. 
Let $F$ be a bounded face of $G_\diamond$, and consider the unique $j=j(F)$ such that $V_j$ has two vertices incident to $F$. Let $v_0$ be the midpoint between these two vertices. 
We define $h(F):=2h_\bullet(F)+h_\circ(F)$ where $h_\bullet(F)$ is the number of vertices of $V_j$ that are above $v_0$ and that are matched to the left, and $h_\circ(F)$ is the number of 
vertices of $V_j$ that are below $v_0$ and that are matched to the right.

We now let $h_\emptyset$ be the height function corresponding to the minimal matching.
If $\mathfrak{m}$ is a pure admissible matching different from the minimal one, then its \emph{reduced} height function $h_{\text{red}}:=h-h_\emptyset$ has a positive maximum. 
We now choose a face $F$ of $G_\diamond$ as follows:
\begin{itemize}
\item[1.] $h_{\text{red}}(F)$ is a maximum of the function $h_{\text{red}}$.
\item[2.] among faces satisfying 1., $h(F)$ is a maximum of the function $h$.
\item[3.] among faces satisfying 1. and 2., $F$ is one that maximizes the quantity $d(F)-u(F)$, where $u(F)$ is the number of $i\leq j(F)$ such that $\diamond_i\in\{\prec, \prec'\}$, where $d(F)=j(F)-u(F)$, and where $j(F)$ is as above the unique index such that the face $F$ is incident to two vertices of $V_{j(F)}$.
\end{itemize}
We now claim that the face $F$ is flippable, and that the corresponding flip is decreasing. This statement follows from a case by case analysis, distinguishing according to the nature of the face $F$, i.e.\ according to the  possible choices of two symbols in $\{\prec, \succ, \prec', \succ'\}$ that give rise to the face $F$. We leave this verification to the reader, and observe that it suffices to prove the lemma.
\end{proof}

\begin{figure}[htpb]
  \centering
  \includegraphics[scale=0.7]{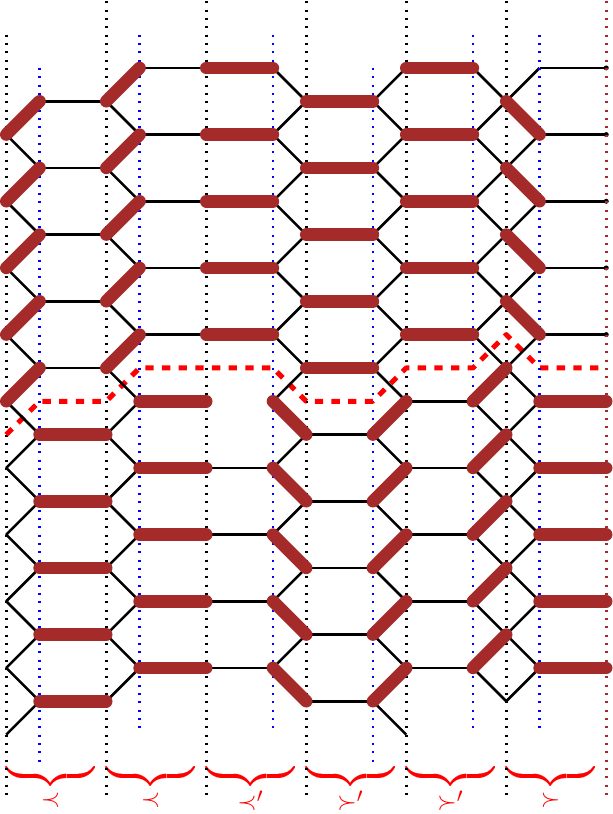}
 \caption{The minimal matching of the graph of Figure~\ref{fig:extendedgraph}. 
Below the red dotted line representing $Y$, all the vertices in $V_0\cup V_1
\cup \dots \cup V_k$ are matched to the left.}
  \label{fig:extendedminimal}
\end{figure}

\subsection{Enumerative results}
\label{sec:extenum}
The enumeration of elements of $\mathcal{S}_\diamond$ can be performed easily via the vertex operator formalism, exactly as we have proceeded for the special case treated in Section~\ref{sec:vertex}. Theorem~\ref{thm:extendedBij} thus implies:
\begin{thm}\label{thm:extendedCounting}
Let $\diamond$ be a word on the alphabet $\{\prec,\succ,\prec',\succ'\}$, and let $T_\diamond(q)$ be the
generating function of admissible matchings of the graph $G_\diamond$ that can
be obtained from the minimal matching by a finite sequence of flips, where the exponent of the variable $q$ marks the length of a minimal such sequence. Then one has
\begin{equation}
  T_\diamond(t) =
  \prod_{\substack{i<j \\ \diamond_i\in\{\prec,\prec'\} \\
      \diamond_j\in\{\succ,\succ'\}}}
  (1+\epsilon_{i,j}q^{j-i})^{\epsilon_{i,j}},
\end{equation}
where $\epsilon_{i,j}=\begin{cases}
    1 & \text{if ~ }
(\diamond_i,\diamond_j)\in\{(\prec,\succ'),(\prec',\succ)\},\\
    -1 & \text{if ~ }
(\diamond_i,\diamond_j)\in\{(\prec,\succ),(\prec',\succ')\}.
\end{cases}
$ 
\end{thm}
\begin{thm}\label{thm:extendedCountingWithStanleyWeights}
Let $\diamond\in\{\prec,\succ,\prec',\succ'\}^k$, and let
$S_\diamond(x_1,x_2,\dots,x_{k-1})$ be the generating function of admissible
matchings of the graph $G_\diamond$ that can be obtained from the minimal
matching by a finite sequence of flips, where the exponent of the variable $x_i$
marks the number of flips at abscissa $\frac{3}{2}i$ in a sequence of minimal length. Then one has
\begin{equation}
  S_\diamond(x_1,x_2,\dots,x_{k-1}) = 
  \prod_{\substack{i<j \\ \diamond_i\in\{\prec,\prec'\} \\
      \diamond_j\in\{\succ,\succ'\}}}
  (1+\epsilon_{i,j}x_ix_{i+1}\dots x_{j-1})^{\epsilon_{i,j}},
\end{equation}
where $\epsilon_{i,j}$ is as in Theorem~\ref{thm:extendedCounting}.
\end{thm}

\subsection{Interpretation as tilings}\label{sec:extendedtilings}
Admissible matchings of the graph $G_\diamond$ can also be interpreted as
tilings. Let $G_\diamond'$ be the dual graph of $G_\diamond$, i.e.\ the
graph with one vertex inside each bounded face of $G_\diamond$, and edges
representing face adjacencies in $G_\diamond$. Then each edge $e$ in a
matching of $G_\diamond$ can be interpreted as a tile in a tiling of
$G_\diamond'$, made by the union of the two faces of
$G_\diamond'$ corresponding to the two endpoints of $e$. In order to display this interpretation, one must first
choose a way of drawing the graph $G_\diamond'$. There is not necessarily a
good canonical way to do this, since for example the graph $G_\diamond'$ can
have multiple edges (this is the case for the graph $G_\diamond$ of
Figure~\ref{fig:extendedgraph}). In any case, it is always possible to fix some
drawing of the dual graph $G_\diamond'$, possibly with
broken lines instead of straight lines in order to take multiple edges into
account.
Figure~\ref{fig:extendedastiling} displays as a tiling of the matching of
Figure~\ref{fig:extendedmatching}. (Other displays would have been possible, would the underlying drawing of the dual graph have been different.)

\begin{figure}[htpb]
  \centering
  \includegraphics[scale=0.5]{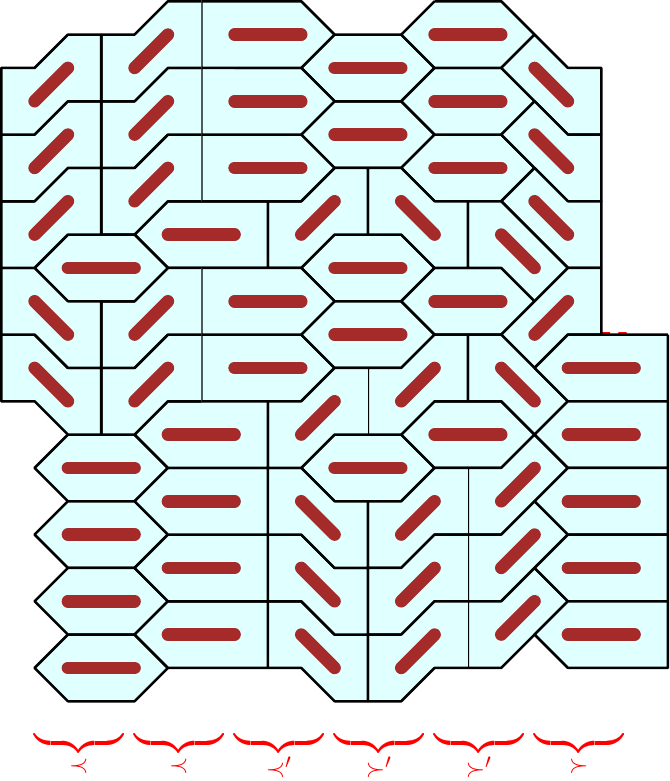}
 \caption{The admissible matching of Figure~\ref{fig:extendedmatching}
represented as a tiling. To make the comparison easier, we have drawn inside
each (blue) tile the corresponding (brown) edge of $G_\diamond$.}
  \label{fig:extendedastiling}
\end{figure}

\subsection{Vertex contraction, and the special case of steep tilings}
\label{sec:extreduc}
If $G$ is a graph and $v$ is a vertex
of degree $2$ of $G$ the \emph{contracted graph} $G//v$ is obtained from $G$ by
contracting the two edges incident to $v$ (thus identifying $v$ with its two
neighbours). See Figure~\ref{fig:contraction}. It is well known, and easy to see, that perfect matchings of $G$
are in bijection with perfect matchings of $G//v$.
\begin{figure}[htpb]
  \centering
  \includegraphics[scale=0.55]{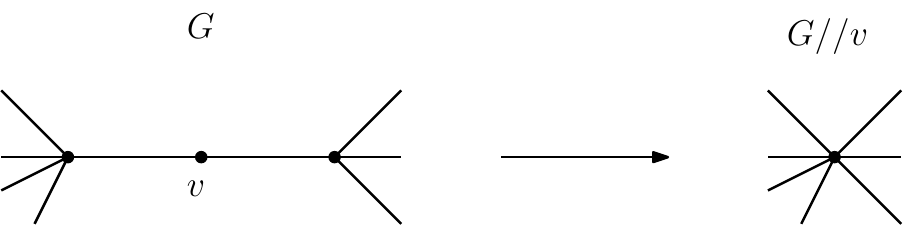}
 \caption{Contracting a vertex of degree $2$.}
  \label{fig:contraction}
\end{figure}
Now, because of the construction of $G_\diamond$, if there is a
$j\in\{1,2,\dots,k-1\}$ such that $\diamond_j\in \{\prec,\succ\}$ and
$\diamond_{j+1}\in \{\prec',\succ'\}$, the vertices in $V_j$ all have degree
$2$. We can thus contract all of them according to the procedure above. If we apply
this construction for all such values of $j$,  we are left with a graph
that we denote by $\widetilde
G_\diamond$. Note that the vertex contraction
removes all the bounded faces of degree $8$, so that $\widetilde G_\diamond$ only
has bounded faces of degree $4$ or $6$. 

Now, as noted at the beginning of this section, steep tilings correspond to the
case where $k$ is even (say $k=2\ell$), and  $\diamond_j\in \{\prec,\succ\}$ for $j$ odd and
$\diamond_j\in \{\prec',\succ'\}$ for $j$ even. In this case, all the vertices
in $V_j$ for $j$ odd can be contracted. One easily sees that in this case the
graph $G_\diamond$ has only bounded faces of degree $4$ and $8$, so that the
graph $\widetilde G_\diamond$ has only faces of degree $4$, see
Figure~\ref{fig:extendedrecoversteeptilings}. More precisely, $\widetilde
G_\diamond$ is the portion of a square lattice that intersects an
oblique strip of width $2\ell$, up to a rotation of $45^\circ$. 
Since the square lattice is self-dual, the corresponding tiles will be
$2\times1$ or $1\times 2$ dominos (after the rotation of $45^\circ$). We thus recover steep tilings as they
were defined in the first part of this paper. 

\begin{figure}[htpb]
  \centering
  \includegraphics[width=\textwidth]{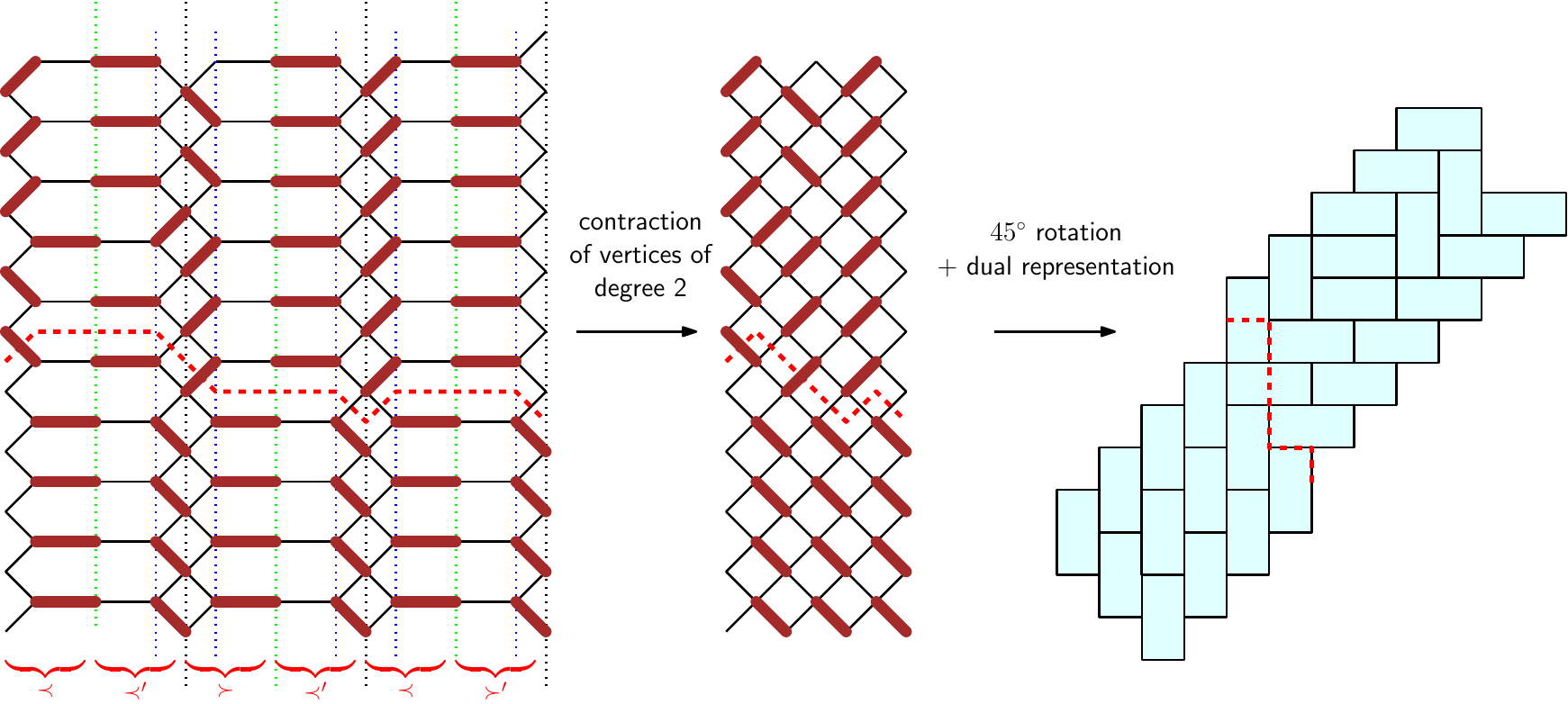}
 \caption{(a) A configuration from the extended model, corresponding to the word $\diamond=++'-+'+-'$. All the vertices on the green dotted lines have degree 2. (b) After contraction of these vertices of degree 2, we obtain a portion of a (rotated) square lattice, equipped with a matching. (c) By rotating the lattice by $45^\circ$ and looking at the dual representation in terms of dominos, we obtain a steep tiling of the oblique strip, as defined in Section~\ref{sec:tilings}.}
  \label{fig:extendedrecoversteeptilings}
\end{figure}

\subsection{The special case of plane partitions}
\label{sec:extpp}
In the case of plane partitions, the correspondence presented in this section
is well known, see e.g.\ \cite{OkounkovReshetikhin:schurProcess}. A plane partition (of width $2\ell$) can be defined as a sequence of
partitions $(\lambda^{(i)})_{-\ell\leq i\leq\ell}$ such that $\lambda^{(-\ell)}=\lambda^{(\ell)}=\emptyset$, and
$\lambda^{(i)}\prec\lambda^{(i+1)}$ if $i<0$ and 
$\lambda^{(i)}\succ\lambda^{(i+1)}$ if $i>0$. 
In our setting it is a pure element of $\mathcal{S}_\diamond$ for
$\diamond=\prec^\ell\succ^\ell$. By applying the construction of this section,
one recovers a standard bijection between plane partitions of width $2\ell$
a family of perfect matchings of the hexagonal lattice, see
Figure~\ref{fig:PP}(b). The dual interpretation, 
in terms of tilings by rhombi,
is also standard, see Figure~\ref{fig:PP}(c).
\begin{figure}[htpb]
  \centering
  \includegraphics[width=\textwidth]{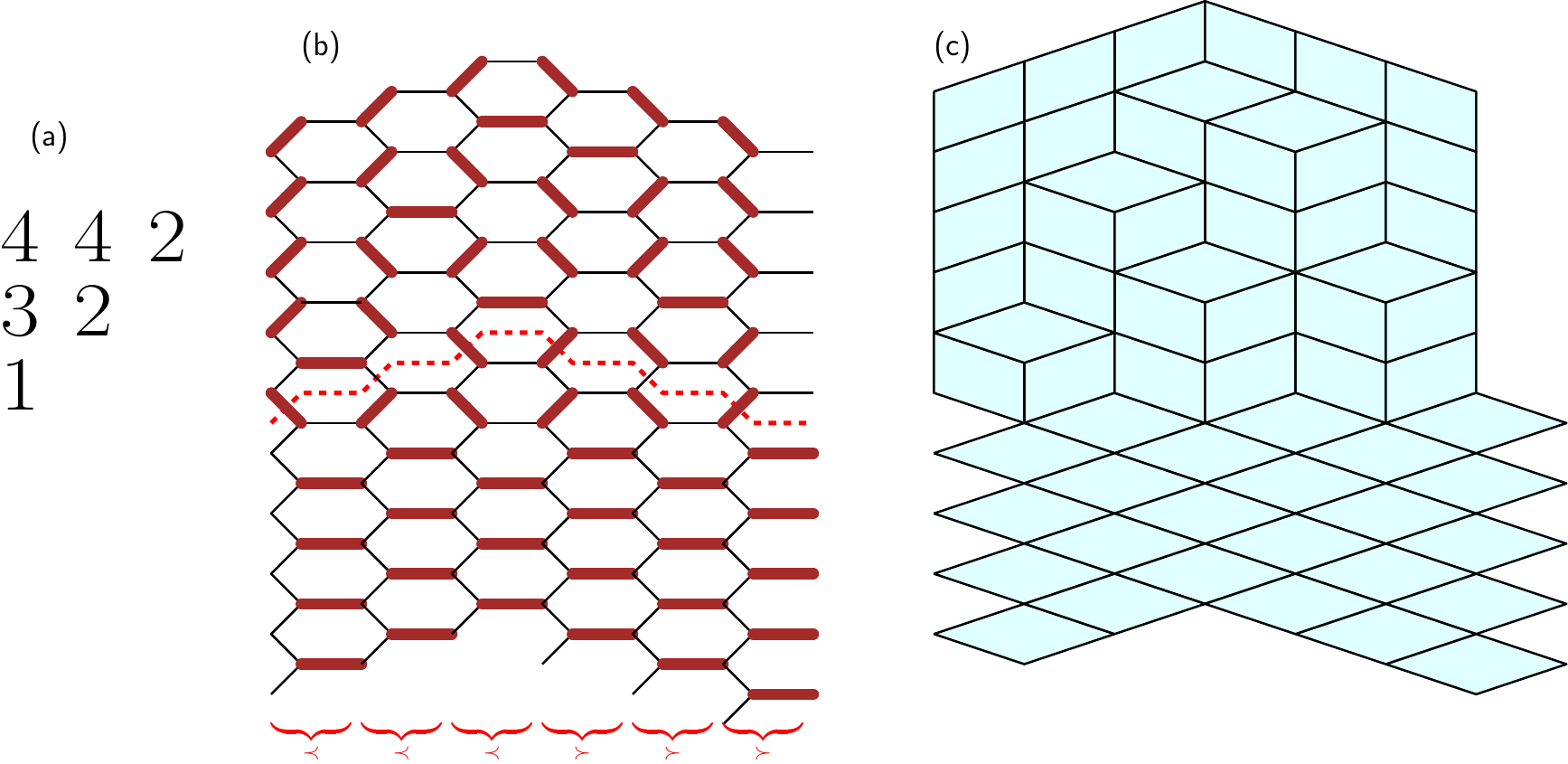}
 \caption{(a) A plane partition corresponding to $\lambda^{(-3)}=\lambda^{(3)}=\emptyset$,
$\lambda^{(-2)}=(1)$,
$\lambda^{(-1)}=(3)$, $\lambda^{(0)}=(4,2)$, $\lambda^{(1)}=(4)$,
$\lambda^{(2)}=(2)$, displayed as an array of numbers. (b) The corresponding matching on the hexagonal lattice. (c) The dual picture, as a rhombi tiling, that can also be viewed as a projected three-dimensional display of the array of (a).}  \label{fig:PP}
\end{figure}

\section{Conclusion and discussion} \label{sec:conc}

In this paper, we have introduced the so-called steep tilings. We have
studied their combinatorial structure and their various avatars, and
obtained explicit expressions for their generating function with
arbitrary asymptotic data and different types of boundary conditions
(pure, mixed, free and periodic). We have also introduced an extended
model interpolating between domino and rhombus tilings, where similar
expressions can be found. We now list a few concluding remarks, some
of which indicate directions for further research.

First, our derivation of the generating functions was done using the
vertex operator formalism, which yield compact proofs. It is however
not very difficult to convert these into \emph{bijective} proofs, for
instance by looking for bijective proofs of the commutation relations
\eqref{eq:gamcom} which form the heart of our derivation. This is
actually related to the celebrated Robinson-Schensted-Knuth
correspondence and, more precisely, its reformulation in terms of
growth diagrams introduced by Fomin. There is an important amount of
literature devoted to that subject, see e.g.\ \cite{Kra06} and
references therein. In particular, the sequences of interlaced
partitions that we consider in this paper are sometimes called
``oscillating supertableaux'' \cite{PP96}. Interestingly, the
bijective approach yields efficient random generation algorithms for
the perfect sampling of steep tilings \cite{BBBCCV}.

Second, beyond the computation of the partition function done in this
paper, we have access to detailed statistics of random steep tilings,
namely the probabilities of finding dominos of a given type at given
positions. In the case of pure boundary conditions, the correlations
of the associated particle system are known explicitly from the
general results of \cite{OkounkovReshetikhin:schurProcess} and,
remarkably, it is possible to deduce from them an explicit formula for
the inverse Kasteleyn matrix of steep tilings with arbitrary
asymptotic data \cite{BBCCR}.  This actually works in the context of
the extended model of Section~\ref{sec:extendedmodel}, and enables us
to recover in a unified and combinatorial way results from
\cite{BorodinFerrari} for rhombus tilings and from \cite{ChhitaYoung} for
domino tilings of the Aztec diamond. In the case of the mixed or
periodic boundary conditions, it should be possible to perform the
same study as the particle correlations are known
\cite{BR05,B2007}. In the case of free boundary conditions however,
even the particle correlations are (to the best of our knowledge)
unknown. Our vertex operator derivation of the partition function
suggests the possible definition of a ``reflected Schur process'' that
we would like to investigate.

Third, as in the case of plane partitions and domino tilings of the
Aztec diamond, steep tilings display a ``limit shape phenomenon'',
that may be observed experimentally via the above mentioned random
generation algorithms, and studied analytically by considering the
asymptotic behaviour of domino/particle correlations. Generally
speaking, all the questions that have been asked or answered in the
literature about rhombus or domino tilings (such as limit shapes, with
different kinds of scalings, or fluctuations, in the bulk or near the
boundary, etc.) can be asked for steep tilings, which opens a wide
area to be investigated.

Finally, a tantalizing question is whether it is possible to go beyond
the determi\-nantal/Schur/free-fermionic setting to study some refined
statistics on steep tilings. For instance, it is well-known that
domino tilings of the Aztec diamond correspond to the 2-enumeration of
alternating sign matrices (ASMs), and one may ask what an ASM with
different ``asymptotic data'' looks like. Regarding sequences of
interlaced partitions, non determinantal generalizations of Schur
processes are the so-called Macdonald processes \cite{MR3152785}. The
search for possible connections with tilings was initiated by
Vuleti\'{c} in her thesis \cite{MR2471939,MR3054919}, and we wonder
whether this could yield new interesting statistics in the case of
domino tilings. Last but not least, some refined weighting schemes for
domino tilings of the Aztec diamonds (leading to a host of new
fascinating limit shapes) were recently considered
\cite{ChhitaYoung,DFSG14}, which raises the question of a possible
connection with our approach.

\medskip
\noindent{\bf Acknowledgments.} 
We would like to thank Dan Betea, C\'edric Boutillier, Sunil Chhita,
Philippe Di Francesco, Patrik Ferrari, \'Eric Fusy, Sanjay Ramassamy,
Mirjana Vuleti\'{c}, Benjamin Young and Paul Zinn-Justin for helpful
discussions. We also thank the anonymous referee for suggesting
proving Proposition~\ref{prop:cyltilseqbij} via height functions.  JB
acknowledges the hospitality of LIAFA, where most of this work was
done, and of MSRI during the \emph{Random spatial processes} 2012
program.

\bibliographystyle{halpha}
\bibliography{aztec.bib}
\end{document}